\documentclass{amsart}

\usepackage[english]{babel}
\usepackage[all]{xy}
\usepackage{amssymb}
\usepackage{amsmath}
\usepackage{amsfonts}
\usepackage{url}
\usepackage[usenames,dvipsnames]{xcolor}
\usepackage{enumerate}
\usepackage[backref=page]{hyperref}
\hypersetup{colorlinks=true,citecolor=NavyBlue,linkcolor=Brown,urlcolor=Orange}
\usepackage{tikz}
\usetikzlibrary{shapes,arrows,decorations.pathreplacing,backgrounds,positioning,fit,matrix}
\tikzstyle{block} = [rectangle, draw, fill=blue!20, 
    text width=5em, text centered, rounded corners, minimum height=4em]
\usepackage{float}

\usepackage{mathtools}

\renewcommand*{\backref}[1]{}
\renewcommand*{\backrefalt}[4]{%
    \ifcase #1 (Not cited.)%
    \or        (Cited on page~#2.)%
    \else      (Cited on pages~#2.)%
    \fi}

\newtheorem{theorem}{Theorem}[section]
\newtheorem{prop}[theorem]{Proposition}
\newtheorem{lemma}[theorem]{Lemma}
\newtheorem{cor}[theorem]{Corollary}

\theoremstyle{definition}
\newtheorem{definition}[theorem]{Definition}
\newtheorem{rem}[theorem]{Remark}

\newcommand{\IC}{\mathbb{C}}

\newcommand{\IH}{\mathbb{H}}
\newcommand{\IN}{\mathbb{N}}
\newcommand{\IP}{\mathbb{P}}
\newcommand{\IQ}{\mathbb{Q}}
\newcommand{\IR}{\mathbb{R}}
\newcommand{\IZ}{\mathbb{Z}}

\newcommand{\cA}{\mathcal{A}}

\newcommand{\cK}{\mathcal{K}}
\newcommand{\cL}{\mathcal{L}}
\newcommand{\cM}{\mathcal{M}}

\newcommand{\cO}{\mathcal{O}}
\newcommand{\cP}{\mathcal{P}}
\newcommand{\cS}{\mathcal{S}}

\newcommand{\cV}{\mathcal{V}}

\newcommand{\cX}{\mathcal{X}}

\newcommand{\MV}{\mathcal{MV}}

\newcommand{\coloneqq}{:=}

\DeclareMathOperator{\rank}{rank}

\newcommand{\Aut}{\operatorname{Aut}}
\newcommand{\Bir}{\operatorname{Bir}}

\newcommand{\Bs}{\operatorname{Bs}}

\newcommand{\cart}{\ar@{}[dr]|\square}
\DeclareMathOperator{\ch}{ch}

\newcommand{\conjug}{\operatorname{conj}}

\newcommand{\Diff}{\operatorname{Diff}}

\newcommand{\FL}{\operatorname{FL}}
\newcommand{\VFL}{\operatorname{VFL}}
\newcommand{\VFP}{\operatorname{VFP}}
\newcommand{\FP}{\operatorname{FP}}
\newcommand{\Gal}{\operatorname{Gal}}
\newcommand{\GL}{\operatorname{GL}}

\newcommand{\Grass}{\operatorname{Grass}}
\newcommand{\Hdg}{\operatorname{Hdg}}

\newcommand{\Hom}{\operatorname{Hom}}
\newcommand{\id}{\operatorname{id}}
\newcommand{\im}{\operatorname{im}}

\newcommand{\KAut}{\operatorname{KAut}}
\newcommand{\KHdg}{\operatorname{KHdg}}

\newcommand{\Mon}{\operatorname{Mon}}

\newcommand{\NS}{\operatorname{NS}}
\newcommand{\PGL}{\operatorname{PGL}}
\newcommand{\Pic}{\operatorname{Pic}}
\newcommand{\Quot}{\operatorname{Quot}}

\newcommand{\rk}{\operatorname{rk}}

\newcommand{\Spec}{\operatorname{Spec}}
\newcommand{\Stab}{\operatorname{Stab}}
\newcommand{\Sym}{\operatorname{Sym}}
\newcommand{\td}{\operatorname{td}}

\newcommand{\triv}{\operatorname{triv}}

\newcommand{\st}{\, \middle| \,}
\newcommand{\hstar}{{*_h}}
\newcommand{\starbar}{{\bar{*}}}

\newcommand{\set}[1]{\left\{ #1 \right\}}
\newcommand{\pa}[1]{\left( #1 \right)}

\title[Finiteness of real structures on compact hyperk\"ahler manifolds]{Finiteness of Klein actions and real structures on compact hyperk\"ahler manifolds}

\author{Andrea Cattaneo}
\address{Andrea Cattaneo\\ Institut Camille Jordan UMR 5208\\ Universit\'e Claude Bernard Lyon 1\\ 69622 Villeurbanne Cedex, France}
\email{cattaneo@math.univ-lyon1.fr}
\author{Lie Fu}
\address{Lie Fu\\ Institut Camille Jordan UMR 5208\\ Universit\'e Claude Bernard Lyon 1\\ 69622 Villeurbanne Cedex, France}
\email{fu@math.univ-lyon1.fr}

\date{\today}

\thanks{Andrea Cattaneo is supported by the LABEX MILYON (ANR-10-LABX-0070) of Universit\'e de Lyon, within the program ``Investissements d'Avenir'' (ANR-11-IDEX- 0007) operated by the French National Research Agency (ANR) and is member of GNSAGA of INdAM. Lie Fu is supported by ECOVA (ANR-15-CE40-0002), HodgeFun (ANR-16-CE40-0011), LABEX MILYON (ANR-10-LABX-0070) of Universit\'e de Lyon and \emph{Projet Inter-Laboratoire} 2017, 2018 by F\'ed\'eration de Recherche en Math\'ematiques Rh\^one-Alpes/Auvergne CNRS 3490.}

\begin{document}

\begin{abstract}
One central problem in real algebraic geometry is to classify the real structures of a given complex manifold. We address this problem for compact hyperk\"ahler manifolds by showing that any such manifold admits only finitely many real structures up to equivalence. We actually prove more generally that there are only finitely many, up to conjugacy, faithful finite group actions by holomorphic or anti-holomorphic automorphisms (the so-called Klein actions). In other words, the automorphism group and the Klein automorphism group of a compact hyperk\"ahler manifold contain only finitely many conjugacy classes of finite subgroups. We furthermore answer a question of Oguiso by showing that the automorphism group of a compact hyperk\"ahler manifold is finitely presented. 
\end{abstract}

\maketitle 
\setcounter{tocdepth}{1}
\tableofcontents

\newpage
\section{Introduction}\label{sect:Intro}
\subsection{Background\,: real algebraic geometry}

Given a complex algebraic variety $X$, a \emph{real form} of $X$ is an algebraic variety $X_{0}$ defined over the field of real numbers $\IR$ such that $X_{0}\otimes_{\IR}\IC$ is isomorphic to $X$ as complex varieties. Of course, a complex variety can have distinct real forms. The simplest example is probably the complex projective line $\IP^{1}_{\IC}$, which has as non-isomorphic real forms the real projective line $\IP^{1}_{\IR}$ and the conic without real points $T_{0}^{2}+T_{1}^{2}+T_{2}^{2}=0$. More generally, given a fixed dimension, on one hand there is a unique smooth quadric over  $\IC$ up to isomorphism\,; on the other hand, any non-degenerate real quadratic form of the given rank gives rise to a real form of the complex quadric, however they are further distinguished by the absolute value of the signature. Naturally, two real forms $X_{0}$ and $X_{0}'$ are said to be \emph{equivalent} if they are $\IR$-isomorphic.

In real algebraic geometry, one important problem is the classification of all real forms, up to equivalence, of a given complex algebraic variety. It is more convenient to reformulate this problem in terms of \emph{real structures}. For simplicity, let us only consider in the introduction smooth and projective complex varieties so that we can shift to the complex analytic language via the GAGA principle \cite{GAGA}. By definition, a \emph{real structure} on a projective complex manifold is an anti-holomorphic involution\,; and the natural equivalence relation between real structures is the conjugation by a holomorphic automorphism. Note that this definition, as well as the equivalence relation, still makes sense in the larger category of complex manifolds (or even complex analytic spaces). It is easy to see that the datum of a real form is equivalent to that of a real structure and the equivalence relations correspond to each other (\emph{cf.}~\cite[Exercise II.4.7]{Hartshorne} and \cite[Introduction]{Benzerga}).

Two basic questions towards the problem of classification of real structures naturally arise\,: for a given  complex manifold
\begin{description}
\item[(Existence)] Does it admit at all any real structure?
\item[(Finiteness)] Are there only finitely many real structures up to equivalence?
\end{description}

For the first question on the existence, an obvious necessary condition is that the complex manifold should be isomorphic to its \emph{conjugate} (\emph{cf.}~Definition \ref{def:Conj} and Lemma \ref{lemma:Index2}). Indeed, if we consider a class of manifolds varying in a moduli space $\cM$, then we have always a set-theoretic involution on $\cM$ sending a point $[X]$ to the point $[\bar{X}]$ represented by the conjugate manifold, and the locus of those manifolds admitting a real structure is a subset of the fixed locus of this involution.

Once there exists at least one real structure $\sigma:X\to X$ on the complex manifold $X$, we have the following cohomological ``classification'' of real structures due to Borel--Serre \cite{BorelSerre}\,: the set of equivalence classes of real structures on $X$, hence the set of $\IR$-isomorphism classes of real forms of $X$ in the projective setting, is in bijection with the (non-abelian) group cohomology $H^{1}(\IZ/2\IZ, \Aut(X))$, where $\IZ/2\IZ$ is naturally identified with the Galois group $\Gal(\IC/\IR)$, $\Aut(X)$ is the group of holomorphic automorphisms of $X$ and the action of the non-trivial element of $\IZ/2\IZ$ on $\Aut(X)$ is given by the conjugation by $\sigma$.

This cohomological interpretation, together with the finiteness result \cite[Th\'eor\`eme~6.1]{BorelSerre}, allows us to answer the second question on the finiteness of real structures in the affirmative when $\Aut(X)/\Aut^{0}(X)$, the group of components of $\Aut(X)$, is a finite group or an arithmetic group\,: for instance, Fano varieties \cite[D.1.10]{RealEnriques}, abelian varieties (or more generally complex tori) \cite[D.1.11]{RealEnriques}, and varieties of general type \emph{etc.}, in particular, when $\dim X=1$. For the next case where $X$ is a complex projective surface, there is an extensive study carried out mainly by the Russian school (Degtyarev, Itenberg, Kharlamov, Kulikov, Nikulin \emph{et al}). We know that there are only finitely many real structures for del Pezzo surfaces, minimal algebraic surfaces \cite{RealEnriques}, algebraic surfaces with Kodaira dimension $\geq 1$ (\emph{cf.}~\cite{MR1936747}) \emph{etc}. The remaining biggest challenge for surfaces seems to be the case of rational surfaces and in fact more recently, based on \cite[Proposition~2.2]{Diller11} and \cite[Theorem 3.13]{MR3480704}, Benzerga \cite{Benzerga} shows that a rational surface with infinitely many non-equivalent real structures, if it exists, must be a blow-up of the projective plane at at least 10 points and possesses an automorphism of positive entropy, \emph{cf.}~also \cite{McMullen07}. 

It turns out that the answer to the finiteness question is negative in general. The first counter-example is due to Lesieutre in \cite{Lesieutre2017}, where he constructs a $6$-dimensional projective manifold with infinitely many non-equivalent real structures and discrete non-finitely generated automorphism group. Inspired by Lesieutre's work, Dinh and Oguiso \cite{DinhOguiso} show that suitable blowups of some K3 surfaces have the same non-finiteness properties, and hence produce such examples in each dimension $\geq 2$.

The finiteness question for higher-dimensional ($\geq 3$) varieties in general can be very delicate, and apart from the general positive results and the counter-examples mentioned above, it is far from being well-understood (see however the related work in the affine situation \cite{DFMJ} and on quasi-simplicity \cite{MR2099111}). The present work is an attempt to investigate this finiteness question systematically for some higher-dimensional manifolds.

\subsection{Klein actions on hyperk\"ahler manifolds}

Our initial purpose of this paper is to give a positive answer to the question on the finiteness of real structures for an important class of manifolds, called compact hyperk\"ahler manifolds (\emph{cf.}~\cite{Beauville}, \cite{HuyInventiones}). Recall that a compact K\"ahler manifold is called \emph{hyperk\"ahler} or \emph{irreducible holomorphic symplectic}, if it is simply connected and has a nowhere degenerate holomorphic $2$-form which is unique up to scalars. Equivalently, these are the simply connected compact K\"ahler manifolds with holonomy group equal to the symplectic group $Sp(n)$, where $n$ is the half of the complex dimension of the manifold. Compact hyperk\"ahler manifolds are the natural higher-dimensional generalizations of K3 surfaces. By the Beauville--Bogomolov Decomposition Theorem (\cite[Th\'eor\`eme~2]{Beauville}, \cite{Bogomolov}), compact hyperk\"ahler manifolds, complex tori and  (strict) Calabi--Yau varieties, are the fundamental building blocs of compact K\"ahler manifolds with vanishing (real) first Chern class. Our first main result is the following\,:

\begin{theorem}\label{main1}
Any compact hyperk\"ahler manifold has only finitely many real structures up to equivalence. 
\end{theorem}

For K3 surfaces, which are the $2$-dimensional hyperk\"ahler manifolds, the work of Degtyarev--Itenberg--Kharlamov \cite[Appendix D]{RealEnriques} not only shows the finiteness of real structures for K3 surfaces but actually gives much stronger results in the broader setting of so-called Klein actions. Let us recall the definition\,: A \emph{Klein automorphism} is a holomorphic or anti-holomorphic diffeomorphism and a \emph{Klein action} on a complex manifold is a group action by Klein automorphisms (Definition~\ref{def:Klein}). We will only consider finite group Klein actions in this paper. Two finite group Klein actions are considered to be  equivalent if they are conjugate by a Klein automorphism of the complex manifold. In the case of K3 surfaces, we have the following result\,:

\begin{theorem}[{\cite[Theorem~D.1.1]{RealEnriques}}]
A complex K3 surface, projective or not, admits only finitely many faithful finite group Klein actions up to equivalence.
\end{theorem}

Our second main result generalizes the previous theorem for higher-di\-men\-sion\-al hyperk\"ahler manifolds\,:

\begin{theorem}\label{main2}
Any compact hyperk\"ahler manifold has only finitely many faithful finite group Klein actions up to equivalence. 
\end{theorem}

Theorem \ref{main1} will be deduced from Theorem \ref{main2} (\S \ref{sect:FiniteReal}). For Theorem \ref{main2}, what we actually prove is the following stronger result, whose part concerning the Klein automorphism group is equivalent to Theorem \ref{main2}, by Remark \ref{rem:bridge} and Lemma \ref{lemma:FiniteConj}.

\begin{theorem}[Theorems \ref{main2-proj}  $\&$ \ref{main2-non-proj}]\label{main2'}
For a compact hyperk\"ahler manifold, the Klein automorphism group, the automorphism group, as well as the birational automorphism group, contain only finitely many conjugacy classes of finite subgroups.
\end{theorem}

To prove Theorem \ref{main2'} we will distinguish the projective case (\S \ref{sect:Proj}) and the non-projective case (\S \ref{sect:Non-proj}), and the proof for each case does not apply to the other. In the projective case, the geometry of the ample cone (resp.~movable cone) will play a crucial role\,: it defines a non-degenerate convex cone in the space $\NS(X)_\IR$, upon which the (Klein) automorphism group (resp.~birational automorphism group) acts. We will then use results from convex geometry to deal with such actions (\S\ref{sect: convex geometry}) and combine them with the recent work by Amerik--Verbitsky on the so-called Morrison--Kawamata cone conjecture (\S\ref{sect: cone conjectures}). In the non-projective case we will approach the problem from the point of view of non-abelian group cohomology, which will be reviewed in \S\ref{sect: group cohomology}.

We give also some rudimentary results towards the existence of real structures on hyperk\"ahler manifolds in \S \ref{sect: existence}. A complex manifold $X$ admitting a real structure is isomorphic to its conjugate $\bar{X}$, and due to Verbitsky's Global Torelli Theorem for hyperk\"ahler manifolds \cite{VerbitskyTorelli}, we give a description of the periods of those hyperk\"ahler manifolds bimeromorphic to their conjugate, see Proposition \ref{prop:BirToConj}. Furthermore, extending the Torelli Theorem of Markman \cite{Markman}, we provide Theorem \ref{thm:TorelliAnti} as a Hodge-theoretic characterization of those hyperk\"ahler manifolds which admit anti-holomorphic automorphisms.

Various examples of real structures on compact hyperk\"ahler manifolds are constructed in \S \ref{subsect:Examples}: Hilbert schemes and more generally moduli spaces of stable sheaves on K3 surfaces, generalized Kummer varieties and more generally the Albanese fibers of moduli spaces of stable sheaves on abelian surfaces, Fano varieties of lines on cubic fourfolds and Debarre--Voisin hyperk\"ahler fourfolds \emph{etc}.

\subsection{Finite presentation of automorphism groups}\label{sect: finite presentation intro}

Thanks to the work of Sterk \cite{Sterk}, it is known that the automorphism group of a projective K3 surface is always finitely generated, \emph{cf.}~\cite[Corollary~\textbf{15}.2.4]{HuyK3}. We ask whether this finiteness property also holds for automorphism groups, or bimeromorphic automorphism groups, of all compact hyperk\"ahler manifolds.

On one hand, in the non-projective case, the following result of Oguiso provides a quite satisfying and precise answer\,:
\begin{theorem}[\cite{Oguiso}]\label{thm:Oguiso}
Let $X$ be a \emph{non-projective} compact hyperk\"ahler manifold. Then its group of bimeromorphic automorphisms $\Bir(X)$ is an almost abelian group of rank at most $\max\{1, \rho(X)-1\}$, where $\rho(X)$ is the Picard rank of $X$. Hence the same conclusion holds for the automorphism group $\Aut(X)$ as well. In particular, $\Bir(X)$ and $\Aut(X)$ are finitely presented. 
\end{theorem}

Here an \emph{almost abelian group of rank $r$} means a group isomorphic to $\IZ^{r}$ up to finite kernel and cokernel, see \cite[\S 8]{Oguiso} for the precise definition.

\bigskip
On the other hand, for a \emph{projective} hyperk\"ahler variety $X$, $\Aut(X)$ and $\Bir(X)$ are of more complicated nature. For example, in \cite{MR2231119} and \cite[Theorem~1.6]{MR2296437}, Oguiso shows that these two groups are not necessarily \emph{almost abelian}, \emph{i.e.}~abelian up to finite kernel and cokernel (see \cite[\S 8]{Oguiso}). Nevertheless, using Global Torelli Theorem (\cite{VerbitskyTorelli}, \cite{Markman}, \cite{HuybrechtsTorelli}), Boissi\`ere and Sarti  \cite[Theorem~2]{MR2957195} prove that $\Bir(X)$ is finitely generated.
The finite-generation problem for $\Aut(X)$ remained open ever since (\cite[Question~1.5]{MR2231119}, \cite[Question~1]{MR2957195}).
Our third main result is to give this question an affirmative, and stronger, answer\,:

\begin{theorem}\label{main3}
For any projective hyperk\"ahler manifold $X$, the automorphism group $\Aut(X)$ and the birational automorphism group $\Bir(X)$ are finitely presented and satisfy $(\FP_{\infty})$ property.
\end{theorem}

This result contrasts to the examples of Lesieutre \cite{Lesieutre2017} and Dinh--Oguiso \cite{DinhOguiso} mentioned above.
See \S \ref{sect:FiniteGeneration} for the notion of $(\FP_{\infty})$ property and the proof of Theorem \ref{main3}. 

\begin{rem}
After the first version of our paper appeared on arXiv, the preprint \cite{KurnosovYasinsky} of Kurnosov and Yasinsky provided another proof of Theorem \ref{main2'} and Theorem \ref{main3}, with the key point being again the cone conjectures. The main difference is that Looijenga's results in convex geometry we used here are replaced by results on the geometry of CAT(0)-spaces.
\end{rem}

\smallskip
\noindent\textbf{Notation and convention\,:} 
\begin{itemize}
\item For a complex manifold, an \emph{automorphism} is always holomorphic unless we say explicitly \emph{anti-holomorphic} or \emph{Klein}.
\item As we will deal a lot with maps and composition of maps, we will drop the composition symbol $\circ$ sometimes. So $fg$ means $f \circ g$, \emph{i.e.}~$(fg)(x) = f(g(x))$.
\item A map between two complex vector spaces is called anti-linear or $\overline\IC$-linear, if it is $\IR$-linear and anti-commutes with the multiplication by $\sqrt{-1}$.
\end{itemize}

\smallskip
\noindent\textbf{Acknowledgements\,:} We are grateful to Ekaterina Amerik, Samuel Boissi\`ere, Kenneth Brown, Gr\'egoire Menet, Giovanni Mongardi and Jean-Yves Welschinger for helpful discussions. The work started during the second Japanese-European Symposium on symplectic varieties and moduli spaces at Levico Terme in September 2017. We would like to thank the organizers and other participants of the conference.

\section{Klein automorphisms and real structures}

As alluded to in the introduction, anti-holomorphic automorphisms will play an equally important role as holomorphic ones in real algebraic geometry. We start with the notion that comprises both.

\begin{definition}[Klein automorphisms, \emph{cf.}~\cite{RealEnriques}]\label{def:Klein}
Let $X$ be a complex manifold and  $G$ be a group.
\begin{itemize}
\item A \emph{Klein automorphism} of $X$ is a holomorphic or anti-holomorphic diffeomorphism from $X$ to itself.
We denote by $\KAut(X)$ the group of Klein automorphisms of $X$. The (biholomorphic) automorphisms of $X$ naturally form a normal subgroup $\Aut(X)$, which is of index at most two in $\KAut(X)$.
\item A \emph{Klein action} of $G$ on $X$ is a group homomorphism $\rho: G \longrightarrow \KAut X$. We say that $\rho$ is \emph{faithful} if it is injective.
\item Two Klein actions $\rho_{1}, \rho_{2}$ of $G$ on $X$ are said to be \emph{conjugate}, if there exists a Klein automorphism $f\in \KAut(X)$ such that $\rho_{1}(g)=f\circ \rho_{2}(g)\circ f^{-1}$ for all $g\in G$.
\end{itemize}
\end{definition}

To understand $\KAut(X)/\Aut(X)$, let us recall the following standard operation\,:

\begin{definition}[Conjugate manifold]\label{def:Conj}
Given a complex manifold $X=(M, I)$, with $M$ being the underlying differentiable manifold and $I$ being the complex structure, the \emph{conjugate} of $X$ is the complex manifold $\bar{X}:=(M, -I)$.
We denote by
\[\conjug\colon \bar{X}\to X\]
the `identity' map, which is an anti-holomorphic diffeomorphism.

If moreover $X$ is the analytic space associated to an algebraic scheme defined over $\IC$, then the conjugate of $X$ is the analytic space associated to the conjugate algebraic scheme $\bar{X}$, which is the base-change of $X$ induced by the complex conjugate of the base field $\IC$\,:
\begin{equation*}
\xymatrix{
 \bar{X} \ar[rr]^{\conjug} \ar[d]  \ar@{}[drr]|\square&& X\ar[d]\\
 \Spec \IC \ar[rr]^{\Spec \conjug}&& \Spec \IC
}
\end{equation*}
where the vertical arrows are structure morphisms.
\end{definition}

\begin{lemma}\label{lemma:Index2}
Let $X$ be a complex manifold. The group $\Aut X$ is a normal subgroup of $\KAut X$, of index at most $2$. Hence we have the exact sequence
\begin{equation}\label{eqn:LES}
 1\longrightarrow \Aut(X)\longrightarrow \KAut(X)\xrightarrow{\;\epsilon\;} \set{\pm 1}.
\end{equation}
The index is $2$ (\emph{i.e.}~$\epsilon$ is surjective) if and only if $X$ is isomorphic to its conjugate $\bar{X}$ as complex manifolds.
\end{lemma}
\begin{proof}
The first assertion is clear from the fact that the composition of two anti-holomorphic automorphisms  is holomorphic. As for the second one, the index being $2$ amounts to the existence of anti-holomorphic automorphisms, which is equivalent to the existence of isomorphisms between $X$ and $\bar{X}$, by composition with the map $\conjug: \bar{X}\to X$ in Definition~\ref{def:Conj}.
\end{proof}

As a special case of Klein automorphisms, we have the following classical notion in real algebraic geometry\,:
\begin{definition}[Real structures]\label{def:RealStr}
Let $X$ be a complex manifold. 
\begin{itemize}
\item  A \emph{real structure} is an anti-holomorphic diffeomorphism $\sigma: X\to X$ of order $2$ (\emph{i.e.}~an involution). 
\item Two real structures $\sigma_{1}$ and $\sigma_{2}$ are said to be \emph{equivalent}, if there exists a holomorphic automorphism $f\in \Aut(X)$ such that $\sigma_{1}\circ f=f\circ \sigma_{2}$.
\end{itemize}
\end{definition}
Some examples of real structures are provided in the hyperk\"ahler setting in \S \ref{subsect:Examples}.
\begin{rem}
\label{rm:Equiv=Conj}
Obviously, two equivalent real structures are conjugate as Klein automorphisms in the sense of Definition~\ref{def:Klein}. It is worth mentioning that the converse is also true. Indeed, if $\sigma$ and $\sigma'$ are two real structures such that there exists $f\in \KAut(X)$ satisfying $\sigma=f\circ \sigma'\circ f^{-1}$, then $\sigma$ and $\sigma'$ are conjugate to each other by a holomorphic automorphism (hence equivalent), namely, $f$ itself if $f$ is holomorphic and $\sigma\circ f$ if $f$ is anti-holomorphic. 
\end{rem}

\begin{rem}\label{rem:invariantSubmanifold}
If $X$ is a complex manifold endowed with a real structure $\sigma$, and $Y \subseteq X$ is a complex subvariety such that $\sigma(Y) = Y$, then $\sigma|_Y$ defines a real structure on $Y$.
\end{rem}

As is discussed in the Introduction, the central problems that we want to address in this paper are the existence and  finiteness of real structures up to equivalence, and the finiteness of faithful finite group Klein actions up to conjugacy. See  \S \ref{sect:Intro} for the known cases and counter-examples, as well as the statement of our main results Theorem~\ref{main1} and Theorem~\ref{main2}.

\section{Pull-backs}\label{sect: pull-backs}
To study the group of Klein automorphisms, we have to look at its various natural representations, among which the most important one for us is its action upon the N\'eron--Severi group/space as well as the ample cone inside it. To this end, we treat with some details in this section the notion of \emph{pull-back} of holomorphic vector bundles and Cartier divisors by anti-holomorphic automorphisms so that we have a well-defined action by the whole group of Klein automorphisms.

Throughout this section, $X$ is a compact complex manifold and $f\in \KAut(X)$ is an \emph{anti-holomorphic} automorphism, unless otherwise specified. To avoid confusion, the notation $f^{*}$ is reserved for the usual (differentiable) pull-back.

\subsection{Functions and divisors}
We start by the pull-back of functions. 
Given any open subset $U$ in $X$ and any holomorphic function $g \in \cO_X(U)$ on it, we define the holomorphic function
\[f^\starbar g := \overline{g \circ f}\]
on the open subset $f^{-1}(U)$.

It obviously enjoys the following two properties\,: for any $g_1, g_2 \in \cO_X(U)$ we have
\[f^\starbar (g_1 + g_2) = f^\starbar g_1 + f^\starbar g_2, \qquad f^\starbar (g_1 \cdot g_2) = f^\starbar g_1 \cdot f^\starbar g_2.\]
In other words, 
\[f^\starbar: \cO_X\longrightarrow f_* \cO_X\]
is an anti-linear isomorphism of sheaves of $\IC$-algebras. This definition of $f^{\starbar}$ clearly extends to the sheaf of meromorphic functions without any change.

Next, let us define the pull-back of Cartier divisors.  Let $D = \set{(U_i, g_i)}$ be a Cartier divisor, where $\set{U_{i}}$ is an open cover of $X$ and $g_i$ is a non-zero meromorphic function on $U_i$ such that $g_i / g_j \in \mathcal{O}^{*}(U_i \cap U_j)$ for all $i, j$. Following \cite[Definition~1.1]{Benzerga}, the \emph{holomorphic pull-back} by $f$ of $D$ is the Cartier divisor
\[f^\hstar D = \set{(f^{-1}(U_i), f^\starbar g_i)}.\]
Since $f^\hstar$ is a homomorphism of the group of Cartier divisors on $X$ preserving the subgroup of principal Cartier divisors, it descends to give an isomorphism
\[\begin{array}{rccc}
f^\hstar: & \Pic X & \longrightarrow & \Pic X\\
 & \cL = \cO_X(D) & \longmapsto & f^\hstar \cL \coloneqq \cO_X(f^\hstar D).
\end{array}\]

\subsection{Vector bundles}
An equivalent way to define the holomorphic pull-back via $f$ of a line bundle is to use directly the cocycle that defines it. This approach generalizes to vector bundles. Let $\cV$ be a holomorphic vector bundle on $X$. As $f$ is anti-holomorphic, the differentiable pull-back $f^{*}\cV$ is an anti-holomorphic complex vector bundle. Its \emph{holomorphic pull-back} by $f$, denoted by $f^\hstar \cV$, is by definition the conjugate bundle of $f^{*}\cV$\,:
\[f^{\hstar} \cV :=\overline{f^{*}\cV}.\]
In other words, let $\cV$ be defined on a trivializing open cover $\set{U_\alpha}$ by the cocycle $g_{\alpha \beta}: U_\alpha \cap U_\beta \longrightarrow \GL(r, \IC)$. Then $f^\hstar \cV$ is the holomorphic vector bundle defined on the trivializing open cover $\set{f^{-1}(U_\alpha)}$ by the cocycle $f^\starbar g_{\alpha \beta}=\overline{g_{\alpha \beta}\circ f}$. One checks easily that this construction is independent of the choice of cocycle, \emph{i.e.}~holomorphic pull-back preserves isomorphisms. We have the compatibility that for any $g\in \Aut(X)$,
\[(f \circ g \circ f^{-1})^* (\cV) \cong  (f^\hstar)^{-1}\circ g^* \circ f^\hstar  (\cV).\]
It is well-known that the Chern classes of a complex vector bundle $E$ and those of its conjugate bundle $\overline{E}$ are related by $c_{i}(\overline{E})=(-1)^{i}c_{i}(E)$. This yields that in $H^{2i}(X, \IZ)$,
\[c_{i}\left(f^{\hstar}\cV\right)=c_{i}\left(\overline{f^{*}\cV}\right)=(-1)^{i}f^{*}c_{i}(\cV),\]
where $f^{*}: H^{2i}(X, \IZ)\to H^{2i}(X, \IZ)$ is the map induced by viewing $f$ as the underlying diffeomorphism.
Moreover, there is a natural map for sections
\begin{equation}\label{eq:compatibility of c1}
\begin{array}{rccl}
f^\hstar: & H^0(X, \cV) & \xrightarrow{\;\simeq\;} & H^0(X, f^\hstar \cV)\\
 & s=(s_{\alpha})_{\alpha} & \longmapsto & (\overline{s_{\alpha} \circ f})_{\alpha},
\end{array}
\end{equation}
which is an anti-linear isomorphism. 


\begin{rem}[Variants]\label{rem:Variants}
Note that the above operation of holomorphic pull-backs by anti-holomorphic automorphisms extends naturally to all coherent sheaves. More precisely, given a coherent sheaf $E$ on a complex manifold $X$ with an anti-holomorphic automorphism $f$, one can write $E$ as the cokernel of a morphism between two locally free sheaves $F_{1}\to F_{0}$, then $f^{\hstar}(E)$ is defined to be the cokernel of $f^{\hstar}(F_{0})\to f^{\hstar}(F_{1})$. Even more generally, by taking locally free resolutions, one obtains an auto-equivalence of the bounded derived category $f^{\hstar}: D^{b}(X)\to D^{b}(X)$ which is exact with respect to the standard $t$-structure. 

\end{rem}

Going back to the case of line bundles, the map \eqref{eq:compatibility of c1} on sections allows us to study the rational map associated to the linear system of the holomorphic pull-back of a line bundle\,:

\begin{lemma}[Base loci]
Let $\cL$ be a holomorphic line bundle on a compact complex manifold and $f$ an anti-holomorphic automorphism. Then
\[\Bs|\cL| = f(\Bs|f^\hstar \cL|).\]
\end{lemma}
\begin{proof}
It follows from the simple observation that 
\[s(x) = 0 \Longleftrightarrow  \overline{s(f(f^{-1}(x)))} = 0 \Longleftrightarrow (f^\hstar s)(f^{-1}(x)) = 0\]
for any holomorphic section $s$ of $\cL$ and any point $x$ of the manifold.
\end{proof}

\begin{prop}
Let $\cL$ be a line bundle on a compact complex manifold $X$ and $f$ an anti-holomorphic automorphism. Then we have a commutative square
\[\xymatrix{\overline{X} \ar@{-->}[rr]^-{\conjug\circ\varphi_{|\cL|}\circ \conjug^{-1}} & & \IP\left(\overline{H^0(X, \cL)}\right)\\
X \ar@{-->}[rr]^-{\varphi_{|f^\hstar \cL|}} \ar[u]^{\conjug\circ f}_{\simeq} & & \IP(H^0(X, f^\hstar \cL)), \ar[u]^{\simeq}_{\IP\left(\conjug\circ f^{\hstar}\right)}}\]
where $\varphi$ denotes the rational map associated to a linear system and $\IP$ is the projective space of $1$-dimensional quotients \`a la Grothendieck.
\end{prop}
\begin{proof}
Recall that, by definition, $\varphi_{|\cL|}$ sends a point $x$ to (the class of) the functional of evaluation of sections in $x$, say $ev_x$. Then on one hand we have $ev_{f(x)}$, and on the other hand we find $ev_x \circ \overline{f^\hstar}$. Now, for any section $s$ of $\cL$, it holds that
\[ev_{f(x)}(s) = s(f(x)), \qquad (ev_x \circ \overline{f^\hstar})(s) = (\overline{f^\hstar s})(x) = s(f(x)),\]
which implies the commutativity of the diagram.
\end{proof}

\begin{cor}[Positivity]\label{cor: f-hstar preserves ampleness}
Let $\cL$ be a holomorphic line bundle on a compact complex manifold $X$ and $f$ an anti-holomorphic automorphism of $X$. Then\,:
\begin{enumerate}
\item $\cL$ is base-point free if and only if $f^\hstar \cL$ is so;
\item $\cL$ is (very) ample if and only if $f^\hstar \cL$ is so.
\end{enumerate}
\end{cor}

\begin{rem}\label{rem:invarient line bundle}
Observe that if $\sigma$ is a real structure on a projective manifold $X$, then for any ample line bundle $\cL$ on $X$, $\cL':=\cL \otimes \sigma^\hstar \cL$ is also ample. Moreover, $\cL'$ is preserved by $\sigma^{\hstar}$, hence it comes from an ample line bundle $\cL'_{0}$ defined on $X_{0}$, the real model of $X$ determined by the real structure $\sigma$. Therefore a sufficiently high power of $\cL'$ induces an embedding of $X$ into a projective space in such a way that $\sigma$ is realized as the restriction to (the image of) $X$ of the real structure of the ambient projective space given by the coordinate-wise complex conjugation.

Similarly, on a compact K\"ahler manifold $X$ together with a real structure $\sigma$, if $\omega\in H^{1,1}(X, \IR)$ is a K\"ahler class, then $-\sigma^{*}(\omega)$ is also a K\"ahler class. It is therefore easy to find a $\sigma$-anti-invariant K\"ahler class, for instance $\omega-\sigma^{*}\omega$.
\end{rem}

\begin{prop}\label{prop:InducedRealStructure}
Let $X$ be a complex manifold with a real structure $\sigma$, and let $\cV$ be a holomorphic vector bundle on $X$. Assume that there exists an isomorphism $\varphi: \sigma^\hstar \cV \longrightarrow \cV$, and consider the composition
\[\Phi: H^0(X, \cV) \xrightarrow{\sigma^\hstar} H^0(X, \sigma^\hstar \cV) \xrightarrow{\varphi_*} H^0(X, \cV).\]
If $s \in H^0(X, \cV)$ is such that $\Phi(s) = s$, then the zero locus $V(s)$ of $s$ is invariant under $\sigma$. In particular, if $V(s)$ is smooth then $\sigma|_{V(s)}$ defines a real structure on it.
\end{prop}
\begin{proof}
We just need to prove that if $s(x) = 0$, then $s(\sigma(x)) = 0$. Let $\{ U_\alpha \}$ be a trivializing open covering for $\cV$, over which $s = (s_\alpha)_\alpha$. Given $x \in V(s)$, we have $\sigma(x) \in U_\alpha$ for some $\alpha$, and so
\[s_\alpha(\sigma(x)) = \Phi(s_\alpha)(\sigma(x)) = \varphi_* \pa{\overline{s_\alpha (\sigma(\sigma(x)))}} = \varphi_* \pa{\overline{s_\alpha (x)}} = 0.\]
\end{proof}

\begin{rem}\label{rem:PhiRealStructure}
Note that the map $\Phi$ defined in Proposition \ref{prop:InducedRealStructure} is not necessarily a real structure on $H^0(X, \cV)$. It is one if $\sigma$ and $\varphi$ satisfy certain compatibility in the sense that the following composition
\[\cV_x \xrightarrow{\conjug} \overline{\cV}_x = (\sigma^\hstar \cV)_{\sigma(x)} \xrightarrow{\varphi_{\sigma(x)}} \cV_{\sigma(x)} \xrightarrow{\conjug} \overline{\cV}_{\sigma(x)} = (\sigma^\hstar \cV)_x \xrightarrow{\varphi_x} \cV_x\]
is the identity for every $x \in X$.
\end{rem}

\subsection{Action on the ample cone\,: the dagger operation}\label{subsect:Dagger}
The N\'eron--Severi group of $X$, denoted by $\NS(X)$, is by definition the image of the first-Chern-class map
\[c_{1}\colon \Pic(X)\to H^{2}(X, \IZ).\]
Now for any $f\in \KAut(X)$, holomorphic or anti-holomorphic, we define the \emph{holomorphic pull-back} $f^{\dagger}\,: H^{2}(X, R) \to H^{2}(X,R)$ with $R=\IZ, \IQ, \IR$ or $\IC$ as
\[f^\dagger \coloneqq \epsilon(f)f^{*}=\left\{ \begin{array}{ll}
f^* & \text{if } f \in \Aut X,\\
-f^* & \text{if } f \notin \Aut X.
\end{array} \right.\]
where $f^{*}$ is the usual pull-back by regarding $f$ as a diffeomorphism and $\epsilon$ is the \emph{signature map} in (\ref{eqn:LES}). Obviously, we have $(f \circ g)^\dagger = g^\dagger \circ f^\dagger$ and $(f^{-1})^\dagger = (f^\dagger)^{-1}$ for every $f, g \in \KAut X$\,; hence the second cohomology of $X$ has a \emph{right} action of the group $\KAut(X)$.

Recall that for a projective complex manifold $X$, its \emph{ample cone} $\cA(X)$ is the (open) convex cone of all ample $\IR$-divisor classes, which sits inside the N\'eron--Severi space\,:
\[\cA(X)\subseteq \NS(X)_{\IR}\subseteq H^{2}(X, \IR).\]

\begin{lemma}\label{lemma:dagger}
The right action by $\dagger$ of $\KAut(X)$ upon $H^{2}(X,\IR)$ preserves the N\'eron--Severi space $\NS(X)_{\IR}$ and the ample cone $\cA(X)$.
\end{lemma}
\begin{proof}
For any $f\in \KAut(X)$, we have
$$f^{\dagger}c_{1}(\cL)=\epsilon(f)f^{*}c_{1}(\cL)=c_{1}(f^{\hstar}\cL),$$
where $f^{\hstar}(\cL)$ is a line bundle, and it is ample if $\cL$ is ample by Corollary~\ref{cor: f-hstar preserves ampleness}. We can conclude since $\NS(X)_\IR$ (\emph{resp.}~$\cA(X)$) is generated as $\IR$-vector space (\emph{resp.}~cone) by the first Chern classes of line bundles (\emph{resp.}~ample line bundles), and so it suffices to check for elements of the form $c_{1}(\cL)$ with $\cL$ being a line bundle (\emph{resp.}~an ample line bundle).
\end{proof}

Switching to a left action by taking the inverse, we get a homomorphism
\begin{eqnarray*}
 \KAut(X)&\longrightarrow& \Aut\left(\NS(X)_{\IR}\right)\\
 f&\mapsto& (f^{-1})^{\dagger},
\end{eqnarray*}
which preserves the ample cone and extends the natural homomorphism $\Aut(X)\to \Aut(\NS(X)_{\IR})$ given by the usual pull-back.

\section{Non-abelian group cohomology}\label{sect: group cohomology}
\subsection{A reminder on group cohomology}
The main reference is \cite{BorelSerre}. Fix a finite group $G$. 
A $G$\emph{-group} is a group $A$ with a (left) $G$-action, that is, a homomorphism $G\to \Aut(A)$. A homomorphism between two $G$-groups is called $G$\emph{-equivariant} or a $G$\emph{-homomorphism} if it commutes with the $G$-action. We hence obtain the category of $G$-groups. 

Taking the $G$-invariant subgroup $A\mapsto A^{G}$ provides a natural functor from the category of $G$-groups to the category of groups. The theory of non-abelian group cohomology is about the functor $H^{1}(G, -)$ from the category of $G$-groups to the category of pointed sets.

Let us briefly recall the definition. For any $G$-group $A$, we write $\prescript{g}{}{a}$ for the result of the action of $g \in G$ on the element $a \in A$. Then
\begin{itemize}
\item The pointed set of $1$\emph{-cocycles} is
\[Z^{1}(G, A):=\left\{\phi: G\to A\mid  \phi(gh)=\phi(g) \cdot \prescript{g}{}{\phi(h)} \right\}\,;\]
with base point being the constant map to the identity of $A$.
\item Two $1$-cocycles $\phi$ and $\psi$ are \emph{equivalent}, denoted by $\phi\sim \psi$, if there exists $a\in A$, such that $a \psi(g)= \phi(g) \cdot \prescript{g}{}{a}$.
\item The first cohomology of $G$ with values in $A$, which is a pointed set, is defined as 
$$H^{1}(G, A):=Z^{1}(G, A)/\sim,$$
with the class of the trivial cocycle as the base point.
\end{itemize}

\begin{rem}[Abelian group cohomology]\label{rem: abelian group cohomology}
In general, $H^1(G, A)$ is only a pointed set instead of a group. However when $A$ is an abelian group (called a $G$-\emph{module}), we see that $Z^1(G, A)$ has a natural structure of abelian group, and the equivalence class of the trivial cocycle defines a subgroup $B^1(G, A)$, called the \emph{coboundaries}. Hence $H^1(G, A)$ can be defined as the quotient abelian group $Z^1(G, A) / B^1(G, A)$. Moreover, in this case, the group cohomology extends to higher degrees. 
\end{rem}

As usual, for a short exact sequence\footnote{This means that the $G$-homomorphism from $A'$ to $A$ is injective and identifies $A'$ with a normal subgroup of $A$ such that the quotient group is isomorphic to $A''$ via the $G$-homomorphism from $A$ to $A''$.} of $G$-groups 
\begin{eqnarray}\label{eqn:SES}
 1\to A'\to A\to A''\to 1,
\end{eqnarray}
there is an exact sequence of pointed sets\footnote{Recall that a sequence of morphisms of pointed sets is called \emph{exact}, if the image of a morphism is equal to the fiber of the next morphism over the base point.} \cite[Proposition 1.17]{BorelSerre}
\begin{equation}\label{eqn:LongES}
 1\to A'^{G}\to A^{G}\to A''^{G}\to H^{1}(G, A')\to H^{1}(G, A)\to H^{1}(G, A''),
\end{equation}

To study the fibers of maps in this exact sequence, we need the following notion which produces a new $G$-group out of an old one. 
\begin{definition}[Twisting, {\emph{cf.}~\cite[$\S$1.4]{BorelSerre}}]\label{def:Twist}
Let $A$ be a $G$-group and $A'$ a normal subgroup of $A$ stable by $G$. Let $A''$ be the quotient $G$-group. Then for any $1$-cocycle $\phi\in Z^{1}(G, A)$, define a new $G$-action on $A'$ by
\begin{eqnarray*}
G\times A'&\to& A'\\
(g, x) &\mapsto& \phi(g) \cdot \prescript{g}{}{x} \cdot \phi(g)^{-1}\,;
\end{eqnarray*}
and a new $G$-action on $A''$ by
\begin{eqnarray*}
G\times A''&\to& A''\\
(g, x) &\mapsto& [\phi(g)] \cdot \prescript{g}{}{x} \cdot [\phi(g)]^{-1}\,;
\end{eqnarray*}
The cocycle condition implies that these are well-defined actions\,; two equivalent $1$-cocycles will define isomorphic $G$-groups. The new $G$-groups are denoted by $A'_{\phi}$ and $A''_{\phi}$ respectively, called the \emph{twisting} by $\phi$ of $A'$ and $A''$.
\end{definition}

Returning to \eqref{eqn:LongES}, by \cite[$\S$1.16]{BorelSerre}, there is a right action of $A''^G$ on $H^1(G, A')$\,: given $a'' \in A''^G$ and $c \in H^1(G, A')$, choose a lift $a$ of $a''$ in $A$ and a representative $\phi$ for $c$ in $Z^1(G, A')$. Then the right action of $a''$ sends $c$ to the class in $H^1(G, A')$ represented by the cocycle $\phi'(g) = a^{-1} \cdot \phi(g) \cdot \prescript{g}{}{a}$. This class is well-defined and independent of the choices involved. The importance of this action is that it can be used to describe the fibres of the last map in \eqref{eqn:LongES}\,:

\begin{lemma}[{\emph{cf.}~\cite[Corollaire~1.18]{BorelSerre}}]\label{lemma:twist}
In the exact sequence (\ref{eqn:LongES}) induced by a short exact sequence (\ref{eqn:SES}), the fiber of the last map through an element of $H^{1}(G, A)$ represented by a $1$-cocycle $\phi\in Z^{1}(G, A)$ is in bijection with the set of orbits of $H^{1}(G, A'_{\phi})$ under the action of ${A''}_{\phi}^{G}$.\\
In particular, if $H^{1}(G, A'')$ is finite and $H^{1}(G, A'_{\phi})$ is finite for any $\phi\in Z^{1}(G, A)$, then $H^{1}(G, A)$ is also finite. 
\end{lemma}

\begin{rem}\label{rmk:fiberbis}
Similarly, if $A'$ is a (not necessarily normal) $G$-subgroup of $A$, then we still get an exact sequence of pointed sets like (\ref{eqn:LongES}) but without the last term and with $A''$ replaced by the pointed set $A/A'$ of left classes \cite[Proposition~1.12]{BorelSerre}. Moreover, each fiber of $H^{1}(G, A')\to H^{1}(G, A)$ has a similar description as in Lemma~\ref{lemma:twist} as the set of orbits of a twisting of $(A/A')^{G}$ under the action of a twisting of $A^{G}$ \cite[Corollaire~1.13]{BorelSerre}.
\end{rem}

\subsection{Cohomological interpretation}
The main interest of introducing the group cohomology is that it `classifies' the real structures up to equivalence. This observation fits into a more general result due to Borel--Serre \cite[2.6]{BorelSerre}. Their statement is in the algebraic setting and holds for any Galois extension\,; while the following version suits us best\,:
\begin{lemma}[\emph{cf.}~{\cite[Proposition 2.6]{BorelSerre}}]\label{lemma:CohomForRealStr}
Let $X$ be a complex manifold. If there exists a real structure $\sigma$ of $X$, then we have a bijection between the set of equivalence classes of real structures and the first cohomology set $H^{1}(\IZ/2\IZ, \Aut(X))$, where the non-trivial element of $\IZ/2\IZ$ acts on $Aut(X)$ by the conjugation by $\sigma$.
\end{lemma}
\begin{proof}
For the sake of completeness, let us explain why this lemma is almost tautological (without using \cite{BorelSerre}). As a $1$-cocycle $\phi: \IZ/2\IZ\to \Aut(X)$ is determined by its image $\phi(\bar 1)$, let us  write $\phi:=\phi(\bar 1)\in \Aut(X)$ by abuse of notation. The $1$-cocycle condition says simply that $\phi\circ \sigma$ is an involution, while two $1$-cocycles $\phi, \psi$ are equivalent if and only if $\phi\circ\sigma$ and $\psi\circ \sigma$ are conjugate by an automorphism of $X$. Now it is clear that the following map
\begin{eqnarray*}
H^{1}(\IZ/2\IZ, \Aut(X))=Z^{1}(\IZ/2\IZ, \Aut(X))/\sim &\xrightarrow{\;\simeq\;} & \set{\text{Real structures on $X$}}/\sim\\
\phi &\mapsto& \phi\circ \sigma,
\end{eqnarray*} 
is a well-defined bijection.
\end{proof}

\begin{rem}\label{rem:bridge}
Let $X$ be a complex manifold and $G$ be a finite group. By definition, we have also a bijection between the set of conjugacy classes of Klein actions of $G$ on $X$ and the cohomology set $H^{1}(G, \KAut(X))$, where $G$ acts trivially on $\KAut(X)$. Therefore, an equivalent formulation of Theorem~\ref{main2} is that \emph{for any compact hyperk\"ahler manifold $X$, we have  
\begin{enumerate}
\item The cardinality of finite group that can act faithfully by Klein automorphisms on $X$ is bounded\,;
\item For any finite group $G$, $H^{1}(G, \KAut(X))$ is finite, where $G$ acts trivially on $\KAut(X)$.
\end{enumerate}}
\end{rem}

\subsection{Some algebraic results}
We prove here some results involving group cohomology that we need in the subsequent sections. 

\begin{lemma}\label{lemma:FiniteConj}
Let $A$ be a group. Then there are only finitely many conjugacy classes of finite subgroups of $A$ if and only if the following two conditions are satisfied\,:
\begin{enumerate}
\item The cardinalities of finite subgroups of $A$ are bounded.
\item For any finite group $G$, $H^{1}(G, A)$ is a finite set, where $A$ is endowed with the trivial $G$-action.
\end{enumerate}
Moreover, if $A$ satisfies this property then so does any subgroup of $A$ of finite index.
\end{lemma}
\begin{proof}
Let us first show the equivalence\,:

For the `if' part\,: on one hand, by condition (1), there are only finitely many possibilities for the isomorphism class of the finite subgroup of $A$. On the other hand, for any fixed abstract finite group $G$, the set of conjugacy classes of subgroups of $A$ with an isomorphism to $G$ is in bijection with the subset of $H^{1}(G, A):=\Hom(G, A)/\sim_{\conjug}$ consisting of classes of injective homomorphisms, hence finite. By forgetting the isomorphisms to $G$, this implies that the set of conjugacy classes of subgroups of $A$ that are isomorphic to $G$ is finite.

For the `only if' part, (1) is clear. For (2), we identify again $H^{1}(G, A)$ with homomorphisms from $G$ to $A$ up to conjugation. To determine such a homomorphism, firstly there are obviously only finitely many possibilities for the kernel\,; secondly, by assumption there are only finitely many possibilities for the image, up to conjugacy; while for each fixed kernel $K$ and image $H\subseteq A$, the set of conjugacy classes of the homomorphisms is in bijection with the finite set of group isomorphisms from $G/K$ to $H$. Therefore, $H^{1}(G, A)$ is finite.

Finally for the last assertion, let $A'$ be a subgroup of $A$ of finite index. Then the condition (1) obviously passes to any subgroup and we only need to check (2) for $A'$. Let $G$ be any finite group, then we have an exact sequence of pointed sets, where $A/A'$ is the (finite) $G$-set of left classes (\cite[Proposition~1.12]{BorelSerre})\,:
\begin{eqnarray}\label{eqn:MiddleES}
 1\to {A'}^{G}\to A^{G}\to (A/A')^{G}\to H^{1}(G, A')\to H^{1}(G, A).
\end{eqnarray}
The last term of \eqref{eqn:MiddleES} being finite by assumption, the finiteness of $H^{1}(G, A')$ is equivalent to the finiteness of fibers of the last map in \eqref{eqn:MiddleES}. Thanks to \cite[Corollaire~1.13]{BorelSerre}, the fiber through an element of $H^{1}(G,A')$ represented by a $1$-cocycle $\phi\in Z^{1}(G, A')$ is in bijection with the set of orbits of $(A_{\phi}/A'_{\phi})^{G}$ under the action of $A^{G}_{\phi}$, where $A_{\phi}$ and $A'_{\phi}$ are the $G$-groups obtained by twistings by $\phi$ (Definition~\ref{def:Twist}, Remark \ref{rmk:fiberbis}). In any case, $A/A'$ is a finite set, hence so are the fibers of the last map in \eqref{eqn:MiddleES}. The finiteness of $H^{1}(G, A')$ is proved.
\end{proof}

The next lemma is known, but we give here a proof for the sake of completeness.

\begin{lemma}\label{lemma:GoodGroups}
Let $G$ be a finite group, and let $A$ be a group endowed with a $G$-action. If $A$ is either a finite group or an abelian group of finite type, then $H^1(G, A)$ is a finite set (regardless of the action of $G$).
\end{lemma}
\begin{proof}
If $A$ is finite, then $H^1(G, A)$ is finite by definition. Assume now that $A$ is a finitely generated abelian group, then $H^{1}(G, A)$ is the quotient of the abelian group of $1$-cocycles $Z^{1}(G,A)$ by the subgroup of $1$-coboundaries $B^{1}(G,A)$, see Remark~\ref{rem: abelian group cohomology}. It is easy to see that the set of all maps $\set{f: G \longrightarrow A}$ is a finitely generated abelian group (which is isomorphic to $A^{|G|}$). Hence so are the subgroups $Z^1(G, A)$ and $B^1(G, A)$. Hence $H^{1}(G,A)$ inherits in a natural way the structure of finitely generated abelian group. Let now $f \in Z^1(G, A)$, and define $x = -\sum_{g \in G} f(g)$\,: we observe that for every $s \in G$ we have the equalities
\[\prescript{s}{}{x} = -\sum_{g \in G} \prescript{s}{}{f(g)} = \sum_{g \in G} (f(s) - f(sg)) = |G| f(s) + x,\]
showing that $|G|f$ is a 1-coboundary. This implies that $H^1(G, A)$ is of torsion, hence finite.
\end{proof}

The following algebraic result is a key gadget needed in the proof of main results.

\begin{lemma}[Filtration]\label{lemma:filtration}
Let $A$ be a group. Assume that there is a finite filtration
\[\set{1} = A_n \subseteq A_{n - 1} \subseteq \ldots \subseteq A_1 \subseteq A_0 = A\]
by normal subgroups of $A$, such that for any $0\leq i\leq n-1$, $A_{i}/A_{i+1}$ is either a finite group or an abelian group of finite type. Then there are only finitely many conjugacy classes of finite subgroups in $A$. Moreover, for any finite group $G$ and any $G$-action on $A$ preserving the filtration, $H^{1}(G, A)$ is finite.   
\end{lemma}

\begin{proof}
By Lemma~\ref{lemma:FiniteConj} and Lemma~\ref{lemma:GoodGroups}, we have for any $0\leq i\leq n-1$, the following two properties
\begin{enumerate}
\item[$(1_{i})$] The cardinalities of finite subgroups of $A_{i}/A_{i+1}$ are bounded.
\item[$(2_{i})$] For any finite group $G$ and \emph{any} action of $G$ on $A_{i}/A_{i+1}$, $H^{1}(G, A_{i}/A_{i+1})$ is finite.
\end{enumerate}
We prove the following two properties, which are trivial for $k=n$, by descending induction on $k$\,:
\begin{enumerate}
\item The cardinalities of finite subgroups of $A_{k}$ are bounded.
\item For any finite group $G$ and any $G$-action on $A_{k}$ that preserves $A_{j}$ for all $j>k$,  $H^{1}(G, A_{k})$ is a finite set.
\end{enumerate}
Assuming these are true for $k=i+1$, let us show them for $k=i$. 

For (1), let $G$ be any finite subgroup of $A_{i}$, then $|G\cap A_{i+1}|$ is bounded by the induction hypothesis (1) for $k=i+1$ and $G/G\cap A_{i+1}$ is a subgroup of $A_{i}/A_{i+1}$, whose cardinality is bounded by $(1_{i})$. Therefore the cardinality of $G$ is bounded. (1) is proved for $k=i$.

For (2), let $G$ be a finite group which acts on $A_{i}$ preserving $A_{j}$ for all $j>i$. The  short exact sequence of $G$-groups
\[\xymatrix{1 \ar[r] & A_{i+1} \ar[r] & A_{i} \ar[r] & A_{i}/ A_{i+1} \ar[r] & 1,}\]
induces an exact sequence of pointed sets
\begin{equation*}
\xymatrix{ \cdots\to (A_{i}/ A_{i+1})^{G} \ar[r] &  H^{1}(G, A_{i+1}) \ar[r] & H^{1}(G, A_{i}) \ar[r] & H^{1}(G, A_{i}/ A_{i+1}).}
\end{equation*}
The last set being finite by $(2_{i})$, the finiteness of $H^{1}(G, A_{i})$ is equivalent to the finiteness of all fibers of the last map in the previous exact sequence. By Lemma~\ref{lemma:twist}, it is enough to show that $H^{1}(G, (A_{i+1})_{\phi})$ is finite for all $\phi\in Z^{1}(G, A_{i})$, where $(A_{i+1})_{\phi}$ is the group $A_{i+1}$ with the $G$-action twisted by the $1$-cocycle $\phi$ (Definition~\ref{def:Twist}). As all subgroups $A_{j}$ are normal in $A_{i}$ for all $j>i$, the $\phi$-twisted $G$-action on $A_{i+1}$ preserves $A_{j}$ for all $j>i+1$, thus by the induction hypothesis (2) for $k=i+1$, $H^{1}(G, (A_{i+1})_{\phi})$ is indeed finite. Therefore $H^{1}(G, A_{i})$ is finite and (2) is proved for $k=i$.

The induction process being achieved, we take $k=0$ and can conclude by invoking Lemma~\ref{lemma:FiniteConj}.
\end{proof}

\begin{rem}\label{rm:NormalityPb}
Apparently, the previous lemma should be compared to \cite[D.1.7]{RealEnriques}, where each subgroup in the filtration is only required to be normal in the precedent one but not necessarily in the ambiant group. However, the authors think the statement in \emph{loc.~cit.}~is flawed at this point\,: the normality inside the whole group is necessary and implicitly used in the proof there (except in the case that $G=\IZ/2\IZ$, the statement in \cite[D.1.7]{RealEnriques} is still true and the proof can be amended by using a conjugate filtration each time). On the other hand, in Lemma~\ref{lemma:filtration} we also allow the constraints on the subquotients of the filtration to be slightly more flexible. Needless to say, the idea of the statement, the proof and the usage of Lemma~\ref{lemma:filtration} are essentially due to \cite[D.1.7]{RealEnriques}.
\end{rem}

\section{Compact hyperk\"ahler manifolds}
Let us now specialize to a particularly interesting class of complex manifolds\,:
\begin{definition}
A \emph{compact hyperk\"ahler manifold} is a compact K\"ahler manifold $X$ such that 
\begin{itemize}
\item $X$ is simply connected\,;
\item $H^{0}(X, \Omega_{X}^{2})=\IC\cdot \eta$ with $\eta$ nowhere degenerate. 
\end{itemize} 
\end{definition}
In particular, it is an even-dimensional complex manifold with trivial canonical bundle. A generic hyperk\"ahler  manifold in the moduli space is non-projective. We refer to \cite{Beauville}, \cite{HuyInventiones}, \cite{GrossHuybrechtsJoyce} and \cite{Markman} for the basic theory of compact hyperk\"ahler manifolds. In this section, we will recall some needed results and extend them to the version that we apply in the proof of main theorems. 

Fix a compact hyperk\"ahler manifold $X$ of complex dimension $2n$. One crucial structure we need is the Beauville--Bogomolov--Fujiki (BBF) quadratic form \cite{Beauville} on $H^{2}(X,\IZ)$.

\subsection{Action on the BBF lattice}\label{sect: action on BBF lattice}
Let $\eta \in H^{2, 0}(X)$ be a generator such that
\[\int_X (\eta \bar{\eta})^n = 1.\]
Then the \emph{Beauville--Bogomolov--Fujiki form} (\cite[\S 8]{Beauville}, \cite{MR514769}, \cite{MR946237}) on the space $H^2(X, \IC)$ is the quadratic form which associates to any $\alpha\in H^{2}(X, \IC)$ the following 
\begin{equation}\label{eqn:BBFform}
 q(\alpha) = \frac{n}{2} \int_X (\eta \bar{\eta})^{n - 1} \alpha^2 + (1 - n) \int_X \eta^{n - 1} \bar{\eta}^n \alpha \cdot \int_X \eta^n \bar{\eta}^{n - 1} \alpha.
\end{equation}
Up to a scalar, this quadratic form induces a non-degenerate integral symmetric bilinear form on $H^{2}(X, \IZ)$ of signature $(3, b_{2}(X)-3)$ (\emph{cf.}~\cite[Part III]{GrossHuybrechtsJoyce}), which makes $H^{2}(X, \IZ)$ a lattice, called the \emph{BBF lattice} of $X$.

\begin{rem}[Isometry]\label{rem:Isometry}
For any $f\in \KAut(X)$, the action $f^{\dagger}$ defined in \S \ref{subsect:Dagger} is an isometry on $H^{2}(X, \IC)$ with respect to the BBF form. In fact, any $f \in \KAut(X)$ is an orientation-preserving diffeomorphism of the underlying differentiable manifold (in the anti-holomorphic case, we use the fact that $\dim X$ is even), and so $f^*$ is an isometry for the BBF form by \cite[Theorem 5.3]{VerbitskyTorelli}. In particular, the BBF lattice $H^{2}(X, \IZ)$ admits a right action of $\KAut(X)$ via $\dagger$.
\end{rem}

\subsection{Torelli theorems for hyperk\"ahler manifolds}\label{sect: existence}

We review some facts on the moduli space of compact hyperk\"ahler manifolds\,: Verbitsky's Global Torelli Theorem and Markman's Torelli Theorem for maps. We will provide an extension of the latter which deals also with anti-holomorphic (or Klein) automorphisms.

Let $X = (M, I)$ be a compact hyperk\"ahler manifold, where $M$ is the underlying differentiable manifold and $I$ the complex structure. Recall that the \emph{period domain} is the complex manifold
\[\Omega = \set{[\sigma] \in \IP(H^2(M, \IC)) \st (\sigma, \sigma) = 0, \, (\sigma, \bar{\sigma}) > 0},\]
where the pairing $(-, -)$ is given by the Beauville--Bogomolov--Fujiki form on $X$.

Denote by $MCG(M) = \Diff(M) / \Diff^{0}(M)$ the \emph{mapping class group} of $M$, where $\Diff(M)$ is the group of orientation-preserving diffeomorphisms of $M$ and $\Diff^{0}(M)$ is its identity component, that is, the group of isotopies. Let
\[\operatorname{Teich}:=\left\{\text{complex structures of K\"ahler type on } M\right\} / \Diff^{0}(M)\]
be the \emph{Teichm\"uller space} of $M$, upon which $MCG(M)$ naturally acts.

Let $\operatorname{Teich}^0$ be the connected component of $\operatorname{Teich}$ to which $X$ (or rather $I$) belongs. Note that $\bar{X}$ (or rather $-I$) also belongs to $\operatorname{Teich}^0$, thanks to the existence of the twistor space. 
Denote by $\operatorname{Teich}_{b}^{0}$ the Hausdorff reduction of $\operatorname{Teich}^{0}$.  Let
\[\begin{array}{rccl}
\tilde\cP: & \operatorname{Teich}_{b}^0 & \longrightarrow & \Omega\\
 & Y = (M, I') & \longmapsto & \IP(H^{2, 0}(Y)),
\end{array}\]
be the period map. One of the most remarkable progress in the study of hyperk\"ahler manifolds is the following
\begin{theorem}[Global Torelli Theorem, \emph{cf.}~{\cite[Theorem 1.17]{VerbitskyTorelli}}]\label{thm:VerbGTT}
Let the notation be as above. The maps $\tilde\cP$ is an isomorphism.
\end{theorem}

Let $MCG^{0}(M)$ be the subgroup of $MCG(M)$ consisting of elements preserving the component $\operatorname{Teich}^{0}$, then $MCG^{0}(M)$ acts on $H^2(M, \IZ)$ preserving the Beauville--Bogomolov--Fujiki form.

\begin{definition}[Monodromy group]
The \emph{monodromy group} $\Mon^{2}:=\Mon^2(X)$ is the image of $MCG^{0}(M)$ in $O(H^2(M, \IZ))$ (\emph{cf.}~\cite[Definition 2.12]{AmerikVerbitsky}). Equivalently, it is the subgroup of $O(H^2(X, \IZ))$ generated by the monodromy transformations in the local systems $R^2 \pi_* \IZ$ where $\pi: \cX \longrightarrow B$ is a deformation of $X$ over a complex base (\emph{cf.}~\cite[Definition 1.1]{Markman} and \cite[Remark.~2.21]{AV-rational}).
\end{definition}

More generally, a \emph{parallel transport operator} is an isomorphism $H^2(X, \IZ) \longrightarrow H^2(Y, \IZ)$ which is induced by parallel transport in the local system $R^2 \pi_* \IZ$ along a path $\gamma$, where $\pi: \cX \longrightarrow B$ is a deformation such that $\cX_{t_0} = X$, $\cX_{t_1} = Y$, and $\gamma$ is a path in $B$ joining $t_0$ with $t_1$ (\emph{cf.}~\cite[Definition 1.1]{Markman}).

The groups $MCG^0(M)$ and $\Mon^2(X)$ naturally act on $\operatorname{Teich}^0_b$ and $\Omega$ respectively. Since the period map $\tilde{P}$ is equivariant with respect to these two actions, we get a homeomorphism between the quotient spaces
\[\cP: \operatorname{Teich}^{0}_{b}/ MCG^{0}(M) \longrightarrow \Omega / \Mon^{2}(X).\]
Observe that the space $\Omega / \Mon^{2}(X)$ is extremely non-Hausdorff\,: as pointed out in \cite[Remark 3.12]{MR3413979}, every two non-empty open subsets intersect.

Consider the natural real structure $\tilde r$ on $\Omega$ defined by $\tilde{r}([\sigma]) = [\bar{\sigma}]$\,: it has an \emph{empty} real locus on $\Omega$, and since it commutes with the action of $\Mon^2(X)$ it defines a homeomorphism $r: \Omega/\Mon^2(X) \longrightarrow \Omega/\Mon^2(X)$ of order $2$. Moreover, $\tilde{r}$ corresponds via the period map to the real structure $\tilde R$ on $\operatorname{Teich}^0_b$ defined by $\tilde{R}([X]) = [\overline X]$, and induces an order-2 homeomorphism $R$ on $\operatorname{Teich}^0_b/MCG^0(M)$. We can think of $r$ and $R$ as the natural real structures on the homeomorphic non-Hausdorff spaces $\Omega/\Mon^2(X)$ and $\operatorname{Teich}^0_b/MCG^0(M)$, which are the so-called \emph{birational moduli space} $\cM^0_b$ in \cite{VerbitskyTorelli}.

The above consideration yields the following characterization.

\begin{prop}\label{prop:BirToConj}
Let $X$ be a compact hyperk\"ahler manifold. Then the following conditions are equivalent\,:
\begin{itemize}
\item  $X$ is bimeromorphic to $\bar{X}$\,;
\item The class of $X$ is a fixed point for $R$ in $\operatorname{Teich}^0_B/MCG^0(M)$\,;
\item The period $\cP(X)$ is a fixed point for $r$ in $\Omega / \Mon^2$.
\end{itemize} 
\end{prop}
\begin{proof}
Just observe that $X$ is bimeromorphic to $\bar{X}$ if and only if the corresponding points in $\operatorname{Teich}^0_b$ are in the same $MCG^0(M)$-orbit.
\end{proof}


To give a more precise description of those hyperk\"ahler manifolds admitting an anti-holomorphic automorphism, we will make use of two other ingredients\,: the twistor space of a hyperk\"ahler manifold and Markman's Torelli Theorem for morphisms.

Let us firstly recall the construction of the twistor space. Let $X = (M, I)$ be a compact hyperk\"ahler manifold as before. Denoting by $g$ a hyperk\"ahler metric compatible with the complex structure, then there exist two other complex structures $J$ and $K$ such that $IJ = K$ and $g$ is K\"ahler with respect to both of them. It turns out that $g$ is K\"ahler with respect to all the complex structures of the form $aI + bJ + cK$ with $a, b, c\in \IR$ and $a^2 + b^2 + c^2 = 1$. The set of such complex structures is then naturally identified with $\IP^1$ and the manifold $M \times \IP^1$ is in a natural way a complex manifold (called the \emph{twistor space} of $X$) with the property that the projection to $\IP^1$ is holomorphic and the fibre over $(a, b, c) \in S^2 = \IP^1$ is the complex manifold $(M, aI + bJ + cK)$.

In \cite{Markman} Markman proved the following Torelli Theorem for maps, which characterizes the isometries arising from pull-back by isomorphisms.

\begin{theorem}[{\emph{cf.}~\cite[Theorem 1.3]{Markman}}]\label{thm:Markman}
Let $X$ and $Y$ be compact hyperk\"ahler manifolds which are deformation equivalent. Let $\varphi: H^2(Y, \IZ) \longrightarrow H^2(X, \IZ)$ be a parallel transport operator, which is an isomorphism of integral Hodge structures. There exists an isomorphism $f: X \longrightarrow Y$ such that $f^* = \varphi$ if and only if $\varphi$ maps a/any K\"ahler class on $Y$ to a K\"ahler class on $X$.
\end{theorem}

We propose the following analogue of Markman's Torelli Theorem \ref{thm:Markman} for anti-holomorphic isomorphisms.

\begin{theorem}\label{thm:TorelliAnti}
Let $X$ and $Y$ be two deformation equivalent compact hyperk\"ahler manifolds, and let $\varphi: H^2(Y, \IZ) \longrightarrow H^2(X, \IZ)$ be an isomorphism. There exists an anti-holomorphic isomorphism $g: X \longrightarrow Y$ such that $g^* = \varphi$ if and only if $\varphi$ satisfies the following conditions\,:
\begin{enumerate}
\item it is a parallel transport operator,
\item it is an isometry for the Beauville--Bogomolov--Fujiki quadratic forms,
\item it is an anti-morphism of Hodge structures, that is, $\varphi\left(H^{2,0}(Y)\right)=H^{0,2}(X)$,
\item $\varphi(\cK_Y) \cap (-\cK_X) \neq \varnothing$.
\end{enumerate}
\end{theorem}
\begin{proof}
Consider the `identity' map between $X$ and $\bar X$\,:
\[\begin{array}{rccc}
\conjug: & \bar{X} = (M, -I) & \longrightarrow & X = (M, I)\\
 & x \in M & \longmapsto & x \in M.
\end{array}\]
The map $\conjug^*: H^2(X, \IZ) \longrightarrow H^2(\bar{X}, \IZ)$ enjoys then the following properties.
\begin{enumerate}
\item It is a parallel transport operator, as it coincides with the parallel transport in the twistor space induced by any path from $-I$ to $I$. The reason is that, as we mentioned, this family is differentiably trivial and $\IP^1$ is simply connected.
\item It is an isometry for the Beauville--Bogomolov--Fujiki form, for the reason that this form is invariant under the exchange of $\eta$ and $\bar{\eta}$.
\item It is an anti-morphism of Hodge structures.
\item $\conjug^*(\cK_X) = -\cK_{\bar{X}}$.
\end{enumerate}

Assume that we are given $\varphi$ with the properties in the statement. First of all, observe that the existence of the twistor space implies that $X$ and $\bar{X}$ are deformation equivalent. Then the parallel transport operator
\[\psi = {\conjug^*}\circ \varphi : H^2(Y, \IZ) \longrightarrow H^2(\bar{X}, \IZ)\]
is a Hodge isometry, so by Theorem \ref{thm:Markman} there exists a holomorphic isomorphism $f: \bar{X} \longrightarrow Y$ such that $f^* = \psi$. This means that $g:=  f \circ\conjug^{-1}: X\to Y$ is an anti-holomorphic isomorphism, such that $g^* =  {\conjug^*}^{-1}  \circ f^*= {\conjug^*}^{-1}  \circ\psi = \varphi$.

The other implication is done in a similar way.
\end{proof}

\begin{rem}
Theorem \ref{thm:TorelliAnti} allows us to reduce the problem of existence of anti-holomorphic automorphisms of a compact hyperk\"ahler manifold $X$ to the problem of existence of anti-Hodge monodromy transformations $\varphi$ on $H^2(X, \IZ)$ anti-preserving K\"ahler classes, which remains challenging even for K3 surfaces. 
\end{rem}

As a consequence, if we define
\[\Mon^2_{\Hdg}(X) = \set{\varphi \in \Mon^2(X) \st \varphi \text{ preserves the Hodge structure of } H^2(X, \IZ)}\]
and
\[\Mon^2_{\KHdg}(X) = \set{\varphi \in O(H^2(X, \IZ)) \st \begin{array}{l}
\varphi \in \Mon^2_{\Hdg}(X) \text{ or } -\varphi \in \Mon^2(X)\\
\text{and } \varphi(H^{2, 0}(X)) = H^{0, 2}(X)
\end{array}}\]
then we can merge Theorem \ref{thm:Markman} and \ref{thm:TorelliAnti} together to have a full characterisation of operators of the form $g^\dagger$.

\begin{cor}\label{cor: torelli for isos}
Let $\varphi \in \Mon^2_{\KHdg}(X)$. Then $\varphi = f^\dagger$ for some $f \in \KAut(X)$ if and only if $\varphi$ sends some K\"ahler class to a K\"ahler class.
\end{cor}
\begin{proof}
Indeed we have $\varphi = g^\dagger = -g^*$ if and only if $g^* = -\varphi$, which by Theorem \ref{thm:TorelliAnti} is equivalent to $\varphi \in \Mon^2_{\KHdg}(X)$ and $\varphi$ sends some K\"ahler class to a K\"ahler class.
\end{proof}

\subsection{Examples of real structures on hyperk\"ahler manifolds}\label{subsect:Examples}

We provide in this subsection some natural constructions of real structures on compact hyperk\"ahler manifolds. 

\subsubsection{Hilbert schemes of K3 surfaces}
Let $S$ be a $K3$ surface equipped with a real structure $\sigma$. We show that for any $n\in \IN$, $\sigma$ induces a natural real structure on $S^{[n]}$, the $n$-th Hilbert scheme (or rather Douady space) of $S$. To this end, the easiest way is to use our Torelli Theorem \ref{thm:TorelliAnti} for anti-holomorphic automorphisms.

By \cite[Proposition 6]{Beauville}, for any $n\geq 2$, there is a Hodge isometry 
\[H^2(S^{[n]}, \IZ) \simeq H^2(S, \IZ) \oplus \IZ \cdot \delta,\]
where $\delta$ is half of the class of the exceptional divisor, hence $(\delta, \delta) = -2(n - 1)$ and $H^{2}(S, \IZ)$ is mapped injectively into $H^{2}(S^{[n]}, \IZ)$ by sending $\alpha$ to the pull-back, via the Hilbert--Chow morphism, of the descent on $S^{(n)}$ of the $\mathfrak{S}_{n}$-invariant class $\alpha^{\times n}$ on $S^{n}$.

Consider the automorphism
\[\varphi = \sigma^* \oplus (-\id): H^2(S, \IZ) \oplus \IZ \cdot \delta \longrightarrow H^2(S, \IZ) \oplus \IZ \cdot \delta.\]
As $\sigma$ is a real structure, $\varphi$ is clearly an isometry, involution and an anti-morphism of Hodge structures (\emph{cf.}~Theorem \ref{thm:TorelliAnti}).  To apply Theorem \ref{thm:TorelliAnti}, let us consider the action of $\varphi$ on the K\"ahler classes. By Remark \ref{rem:invarient line bundle}, there exists a K\"ahler class $\omega\in H^{1,1}(S, \IR)$ such that $\sigma^{*}\omega=-\omega$. The image of $\omega$ in $H^{2}(S^{[n]}, \IR)$ is on the boundary of the K\"ahler cone (\emph{i.e.}~semi-positive), however for a sufficiently small $\varepsilon>0$, $\omega-\varepsilon\delta\in H^{1,1}(S^{[n]}, \IR)$ is indeed K\"ahler and is moreover $\varphi$-anti-invariant. 
Finally, $\varphi$ is orientation-preserving (in the sense of \cite[\S 4]{Markman}) by construction\,; and since it acts on the discriminant lattice of $H^2(S^{[n]}, \IZ)$ as $-\id$, we see that $\varphi$ is a monodromy operator by \cite[Lemma 9.2]{Markman}.

With all the hypotheses of Theorem \ref{thm:TorelliAnti} being fulfilled, it implies that $\varphi = {\sigma_{n}}^*$ for some anti-holomorphic automorphism $\sigma_{n}$ of $S^{[n]}$. Since the only (holomorphic) automorphism of $S^{[n]}$ acting trivially on $H^2(S^{[n]}, \IZ)$ is the identity (\cite[Proposition 10]{Beauville}), we conclude that $\sigma_{n}^2 = \id$, that is, $\sigma_{n}$ is a real structure on $S^{[n]}$.

The geometric description of $\sigma_{n}$ is as expected\,: for any length-$n$ closed analytic subscheme $i:Z\hookrightarrow X$, consider the base-change by the conjugate automorphism of the base field $\IC$:
\begin{equation*}
\xymatrix{
&\bar Z \ar@{^{(}->}[dl]_{\sigma\circ i'}\cart\ar[r]^{\conjug}\ar@{^{(}->}[d]^{i'}& Z\ar@{^{(}->}[d]^{i}\\
S \ar[dr]_{f}&\bar S\ar[d]^{f'}\cart\ar[l]^{\sigma} \ar[d] \ar[r]^{\conjug}& S\ar[d]^{f}\\
&\Spec(\IC) \ar[r]^{\conjug} &\Spec(\IC)
}
\end{equation*}
then define the image of $Z$ by $\sigma_{n}$ to be the length-$n$ closed subscheme $\sigma\circ i': \bar{Z}\hookrightarrow S$. One can check the anti-holomorphicity by looking at the induced morphism on tangent spaces. We leave the details to the reader. Note that this construction generalizes to any complex surfaces.

\subsubsection{Moduli spaces of stable sheaves on K3 surfaces}
An important source of examples of hyperk\"ahler manifolds is provided by the moduli spaces of stable sheaves on K3 surfaces, generalizing Hilbert schemes discussed before. As the K\"ahlerness of the moduli spaces of stable sheaves on non-projective K3 surfaces has not been completely settled yet (\emph{cf.}~\cite{MR3621782}, \cite{Perego17}), we restrict ourselves to the algebraic setting. 

Let $S$ be a projective K3 surface. Let $\widetilde H(S, \IZ)$ be the \emph{Mukai lattice}, that is, the free abelian group $H^{*}(S, \IZ)$ endowed with the \emph{Mukai pairing} given by $(v_{0}, v_{1}, v_{2})\cdot (v_{0}', v_{1}', v_{2}')=v_{1}\cdot v_{1}'-v_{0}v_{2}'-v_{0}'v_{2}$, where $v_{i}, v_{i}'\in H^{2i}(S,\IZ)$. Thanks to the works \cite{Beauville}, \cite{Mukai84}, \cite{MR1395935}, \cite{OGrady97} and \cite{HuyInventiones} \emph{etc}, given a Mukai vector $v=(v_{0}, v_{1}, v_{2})\in \widetilde H(S, \IZ)$ with $v_{1}\in \NS(S)$ primitive and $v_{0}>0$, and a $v$-generic ample line bundle $H$, the moduli space of $H$-semistable sheaves on $S$ with Mukai vector $v$, denoted by $M:=M_{H}(S, v)$, is a projective hyperk\"ahler manifold of dimension $2n:=(v, v)+2$, deformation equivalent to the $n$-th Hilbert scheme of $S$\,; and all sheaves parametrized by $M$ are stable. The objective of this subsection is to show that a real structure on $S$ gives rise to a canonical real structure on $M$, under natural compatibility conditions. 

More precisely, let $\sigma$ be a real structure on $S$ such that $\sigma^{*}(v)=v^{\vee}:=(v_{0}, -v_{1}, v_{2})$ and $\sigma^{*}(c_{1}(H))=-c_{1}(H)$. We claim that $$\sigma^{\hstar}:M\to M$$ is a real structure (see Remark \ref{rem:Variants} for the holomorphic pull-back of a coherent sheaf). The first assumption on Mukai vectors, which says nothing else but $\sigma^{*}(v_{1})=-v_{1}$, implies that for any sheaf $E$ on $S$ with $v(E)=v$, we have $$v\left(\sigma^{\hstar}(E)\right)=\ch\left(\overline{\sigma^{*}E}\right)\cdot \sqrt{\td(S)}=\sigma^{*}\left(v(E)\right)^{\vee}=v(E);$$ while the second assumption on polarization, which says that $\sigma^{\hstar}(H)\simeq H$, implies that the stability condition is preserved, as the slope of a torsion-free sheaf $E$ satisfies $$\mu(\sigma^{\hstar} E)=\frac{c_{1}(\overline{\sigma^{*}E})\cdot c_{1}(H)}{\rk E}=\frac{\sigma^{*}\left(-c_{1}(E)\right)\cdot \left(-\sigma^{*}c_{1}(H)\right)}{\rk E}=\mu(E).$$
In other words, the category of sheaves parametrized by $M$ are preserved, hence $\sigma^{\hstar}$ is an involution on $M$. 

To see that $\sigma^{\hstar}$ is anti-holomorphic, we go back to the GIT construction of $M$ (\emph{cf.}~\cite{HuybrechtsLehn}). Denote $P(t):=v_{0}+v_{2}+t(v_{1}\cdot H)+\frac{t^{2}}{2}v_{0}(H^{2})$ the Hilbert polynomial determined by the Mukai vector $v$. By the boundedness of sheaves of fixed Mukai vector, there exists an integer $m$ such that all sheaves parametrized by $M$ are $m$-regular. Let $V$ be a fixed complex vector space of dimension $P(m)$, the dimension of $H^{0}(S, E\otimes H^{m})$ for any $E$ parametrized by $M$. Let $R$ be the stable locus of the Quot-scheme $\Quot(V\otimes H^{-m}, P)$, upon which $\PGL(V)$ naturally acts, then $M$ is the geometric quotient of $R$ by $\PGL(V)$. Now $\sigma$ induces the following anti-holomorphic involution on the Quot-scheme, denoted by $\tilde\sigma$. Choose an isomorphism $f: H \xrightarrow{\simeq} \sigma^{\hstar}H$ such that the composition $H\xrightarrow{f}\sigma^{\hstar}H\xrightarrow{\sigma^{\hstar}(f)}(\sigma^{\hstar})^{2}(H)=H$ is the identity; this is always achievable by modifying $f$ by a scalar. We define the image by $\tilde\sigma$ of a quotient $[q: V\otimes H^{-m}\twoheadrightarrow E]$ to be the quotient $[q': V\otimes H^{-m}\twoheadrightarrow \sigma^{\hstar}(E)]$ given as the following composition
$$\xymatrix{
 q': V\otimes H^{-m}\ar[rr]^{\id_{V}\otimes f^{-m}}_{\simeq} &&V\otimes \sigma^{\hstar}H^{-m}\ar@{->>}[rr]^{\sigma^{\hstar}(q)}&& \sigma^{\hstar}(E).
}$$
By the hypothesis on $f$, $\tilde\sigma$ is an involution on the Quot-scheme, which is anti-holomorphic by construction. As is explained before, the subscheme $R$ is preserved by $\tilde\sigma$. Moreover, it is clear that the action commutes with the natural action of $\PGL(V)$. Therefore, $\tilde\sigma$ descends to a real structure $\sigma_{M}$ on the GIT quotient $R/\PGL(V)=M$, which maps $[E]$ to $[\sigma^{\hstar}(E)]$ as promised.

\subsubsection{Moduli spaces of stable sheaves on abelian surfaces}
Similarly to the previous two examples using K3 surfaces, one can start instead with abelian surfaces (or more generally two-dimensional complex tori). Let $A$ be an abelian surface, $v$ a primitive Mukai vector with $v_{0}>0$ and $v^{2}\geq 6$, and $H$ a $v$-generic polarization, then by the works \cite{Beauville}, \cite{Mukai84} and \cite{MR1872531} \emph{etc}, the Albanese fibers of $M_{H}(A, v)$, denoted by $K_{H}(A, v)$, is a projective hyperk\"ahler variety of dimension $2n=v^{2}-2$, deformation equivalent to generalized Kummer varieties. Now suppose that $\sigma$ is a real structure on $A$ such that it respects the group structure and anti-preserves $v_{1}$ and $H$. Then the same argument in the case of K3 surface applies and shows that $M_{H}(A, v)$ has a natural real structure, which leaves invariant the (isotrivial) Albanese fibration, hence induces a natural real structure on $K_{H}(A, v)$.

\subsubsection{Beauville--Donagi and Debarre--Voisin hyperk\"ahler fourfolds}
We will start by some general results. Let $V$ be an $n$-dimensional complex vector space endowed with a real structure\,: $V=V_{0}\otimes_{\IR}\IC$.
This real structure naturally induces real structures on tensor functors like $V^*$, $V^{\otimes m}$, $\bigwedge^h V$, $\Sym^t V$, as well as on homogeneous varieties (projective spaces, flag varieties \emph{etc.}) constructed from them.

Let $k\in \IN$, $\Grass(k,V)$ be the Grassmannian variety and $\cS$ the tautological sub-bundle on it. For any partition $\lambda$ of length at most $k$, the corresponding Schur functor gives rise to a homogenous bundle $\mathbb{S}_{\lambda}\mathcal{S}^{*}$. Note that they possess natural real structures since they admit natural real forms, namely the $\IR$-scheme $\Grass(k, V_{0})$, the tautological sub-bundle $\mathcal{S}_{0}$ on it and $\mathbb{S}_{\lambda}\mathcal{S}^{*}_{0}$ respectively; we are in the setting of Proposition \ref{prop:InducedRealStructure} and actually a better one\,: $\sigma$ and $\varphi$ are compatible (Remark \ref{rem:PhiRealStructure}).

Since the Bott isomorphism 
\begin{equation}\label{eqn:Bott}
 H^{0}\left(\Grass(k, V), \mathbb{S}_{\lambda}\mathcal{S}^{*}\right)\cong \mathbb{S}_{\lambda}V^{*}
\end{equation}
is clearly compatible with the induced real structures on both sides, for any real element of $\mathbb{S}_{\lambda}V^{*}$, we obtain a section of the homogenous bundle $\mathbb{S}_{\lambda}\mathcal{S}^{*}$ that is invariant under the real structure, hence by Proposition \ref{prop:InducedRealStructure}, its zero locus inherits a real structure from that of $\Grass(k, V)$.

We now provide two examples where this construction yields a real structure on hyperk\"ahler manifolds. 
\begin{itemize}
\item (Beauville--Donagi \cite{BeauvilleDonagi}.) $n=6, k=2, \lambda=(3)$. Let $V$ be a $6$-dimensional complex vector space and $f \in \Sym^3 V^*$ defines a smooth cubic fourfold $X \subseteq \IP(V)$. The zero locus of the corresponding section via \eqref{eqn:Bott} $$s_f \in H^0(\Grass(2, V), \Sym^3 \cS^*)$$ is then the Fano variety of lines contained in $X$, denoted by $F(X)$, which is a hyperk\"ahler fourfold, deformation equivalent to the Hilbert square of a $K3$ surface \cite{BeauvilleDonagi}. Once we endow $V$ with a real structure and choose $f$ to be a real form, we have a natural real structure on $F(X)$.
\item (Debarre--Voisin \cite{DebarreVoisin}.) $n=10, k=6, \lambda=(1,1,1)$. Let $V$ be a $10$-dimensional complex vector space and $f \in \bigwedge^3 V^*$ a \emph{generic} cubic form. It is shown in \cite{DebarreVoisin} that the zero locus of the corresponding section via \eqref{eqn:Bott} $$s_f \in H^0(\Grass(6, V), \bigwedge^3 \cS^*)$$ is a hyperk\"ahler fourfold, deformation equivalent to the Hilbert square of a $K3$ surface. As soon as $V$ is equipped with a real structure and $f$ is chosen real (it is always possible even with the genericity condition on $f$\,: the real locus of $\bigwedge^{3}V^{*}$ is Zariski dense), we get a natural real structure on the hyperk\"ahler fourfold. 
\end{itemize}

In the same spirit, for polarized K3 surfaces with small degree where a Mukai model is available, one can construct from a real structure on the homogenous data a canonically associated real structure on the K3 surface. We leave the details to the reader.

\section{Cone conjecture for hyperk\"ahler varieties: an extension}

Given any projective complex manifold $X$ we can consider the natural action of the automorphism group $\Aut(X)$ on the rational closure (see Definition \ref{def:RationalClosure}) of the ample cone $\cA(X)$. The \emph{Morrison--Kawamata cone conjecture} predicts that this action admits a fundamental domain which is a (convex) rationally polyhedral cone when the canonical bundle of $X$ is numerically trivial. See \cite[Cone conjecture]{MR1360514}, \cite{MR1468356} for the original source,  \cite{Sterk}, \cite{MR771979}, \cite{MR2682184}, \cite{MR2931877} for the surface case, \cite{Markman} for the movable cone analogue for hyperk\"ahler varieties, \cite{MR2987307} for abelian varieties and \cite{VOPsurvey} for a modern survey as well as results for Calabi-Yau varieties. The cone conjecture, as well as its K\"ahler version, has recently been proved in the case of hyperk\"ahler manifolds by Amerik--Verbitsky \cite[Theorem 5.6]{AmerikVerbitsky}, \cite[Remark 1.5]{AV-IMRN} (see also \cite{MarkmanYoshioka} for the cases of $K3^{[n]}$ and Kummer deformation type). The aim of this section is to prove an extended version of the cone conjecture on the ample cone of a projective hyperk\"ahler manifold with respect to the natural action of the Klein automorphism group $\KAut(X)$ given by the $\dagger$ operation defined in \S \ref{subsect:Dagger}.

Apart from Amerik--Verbitsky's work, we will need some general results in convex geometry which are collected in \S\ref{sect: convex geometry}.

\subsection{Convex geometry and actions on cones}\label{sect: convex geometry}

Let $V$ be a finite-dimensional real vector space with a $\IQ$-structure. 

\begin{definition}[\emph{cf.}~{\cite[Definition 2.1]{Looijenga}}]
Let $\Pi$ be a subset of $V$.
\begin{enumerate}
\item \textbf{(Polyhedra, rational polyhedra).} We say that $\Pi$ is a \emph{polyhedron} if it can be defined as the intersection of finitely many closed half-spaces in $V$. If, with respect to the given $\IQ$-structure, those half-spaces are definable over $\IQ$, then we say that $\Pi$ is a \emph{rational polyhedron}.
\item \textbf{(Faces).} Let $\Pi$ be a polyhedron in $V$ given as the intersection of a finite number of closed half-spaces, whose boundaries give a finite collection of hyperplanes $H_1, \ldots, H_m$. A \emph{face} of $\Pi$ is a subset of the following form 
\[F = \Pi \cap \bigcap_{j \in J} H_j, \qquad \text{with } J \subseteq \{ 1, \ldots, m \}.\]
A polyhedron hence has only finitely many faces. A $1$-dimensional face of a polyhedral cone is called a \emph{ray}. 
Be aware that this convention is different from \cite{AmerikVerbitsky}\,: a \emph{face} here is not necessarily of codimension $1$ \footnote{A codimension-one face would be called a \emph{facet}.}.
\end{enumerate}

\end{definition}

\begin{definition}[Rational closure]\label{def:RationalClosure}
Let $C$ be a non-degenerate open convex cone in a finite dimensional real vector space $V$ and fix a $\IQ$-structure on $V$. We define $C^+$ as the convex hull of $\overline{C} \cap V(\IQ)$. Then $C^+$ is again a convex cone, with $C \subseteq C^+ \subseteq \overline{C}$.
\end{definition}

We are interested in the actions of subgroups $\Gamma$ of $\GL(V)$ which stabilize $C$, hence act on $C$. In particular, we seek for a \emph{fundamental domain} for the action of $\Gamma$ on $C$, that is, a closed subset $D$ with non-empty interior $\mathrm{int}(D)$, such that $\bigcup \{ \gamma \cdot D \,|\, \gamma \in \Gamma \}=C$ and the sets $\{ \gamma \cdot  \mathrm{int}(D) \,|\, \gamma \in \Gamma \}$ are mutually disjoint.


The following finiteness property is used in the proof of Theorem \ref{main2'}.

\begin{prop}[Siegel property, {\cite[Theorem 3.8]{Looijenga}}]\label{prop: Siegel}
Let $C$ be a non-de\-gen\-er\-ate open convex cone in a finite dimensional real vector space $V$ equipped with a fixed $\IQ$-structure. Let $\Gamma$ be a subgroup of $\GL(V)$ which stabilizes $C$ and a lattice in $V(\IQ)$. Then $\Gamma$ has the \emph{Siegel property} in $C^+$\,: if $\Pi_1$ and $\Pi_2$ are polyhedral cones in $C^+$, then the collection $\{ (\gamma \cdot \Pi_1) \cap \Pi_2 \,|\, \gamma \in \Gamma \}$ is finite.
\end{prop}

The following result is a generalization of the classical theory of Siegel sets. Recall that for a cone $C$ in $V$, its \emph{open dual} $C^\circ \subseteq V^*$ is the interior of the cone of those real-valued functionals which are non-negative on $C$.
\begin{theorem}[{\emph{cf.}~\cite[Proposition 4.1 and Application 4.14]{Looijenga}}]\label{thm:Looijenga}
Let $C$ be a non-degenerate open convex cone in a finite dimensional real vector space $V$ equipped with a fixed $\IQ$-structure. Let $\Gamma$ be a subgroup of $\GL(V)$ which stabilizes $C$ and some lattice in $V(\IQ)$. Assume that\,:
\begin{enumerate}
\item there exists a polyhedral cone $\Pi$ in $C^+$ such that $\Gamma \cdot \Pi \supseteq C$\,;
\item there exists an element $\xi \in C^\circ \cap V^*(\IQ)$ whose stabilizer in $\Gamma$ is trivial.
\end{enumerate}
Then $\Gamma$ admits a fundamental domain $\Sigma$ for its action on $C^+$, which is a rational polyhedral cone.
\end{theorem}

We will be interested in the case where $V$ comes from a hyperbolic lattice. Then via the metric, $V$ is identified with $V^{*}$ and $C$ with its open dual $C^{\circ}$. Hence the second assumption of Theorem \ref{thm:Looijenga} will be automatically satisfied thanks to the following fact. This proposition should be known to experts. But as we could not find a reference, a proof is included below for the convenience of the reader.

\begin{prop}\label{prop:TrivialStab}
Let $L$ be a hyperbolic lattice\footnote{A \emph{hyperbolic lattice} is a free abelian group of finite rank endowed with a non-degenerate bilinear form of signature $(1, \rank -1)$.} of rank $m\geq 1$ and $C \subseteq V:= L_\IR$ one of the two connected components of the set $\{ v \in V \,|\, v^2 > 0 \}$. Let $\Gamma$ be a subgroup of $O(L)$ which preserves $C$. Then the set of points $x \in C$ whose stabilizer in $\Gamma$ is trivial is open. In particular, there exists a rational point in $C$ with trivial stabilizer.
\end{prop}
\begin{proof}
Let $x \in C$, and assume that $\gamma \in \Gamma$ fixes $x$\,: as a consequence $\gamma \in O(L) \cap O(x^\perp)$, and so $\Stab_\Gamma(x)$ is finite for every $x \in C$. It then makes sense to speak of the set of points with minimal stabilizer, and we want to prove first of all that this set is open, and then that the minimal stabilizer is trivial\,; which will imply the proposition.

Consider the hyperboloid model $\IH^{m - 1} = \IP(C)$ of the hyperbolic space, on which our group $\Gamma$ acts naturally as a discrete group of isometries with respect to the hyperbolic distance $d$. Let $[x] \in \IH^{m - 1}$ be a point with minimal stabilizer, $r = r(x) = \frac{1}{3} \min_{\gamma \in \Gamma \smallsetminus \Stab_\Gamma(x)} d([x], \gamma \cdot [x])$ and consider the ball $B = B([x], r)$. Then $B \cap \gamma \cdot B \neq \varnothing$ if and only if $\gamma \in \Stab_\Gamma(x)$. Moreover, for any $[y] \in B$ and $\gamma \in \Stab_\Gamma(y)$ we have that $d([y], \gamma \cdot [x]) = d([y], [x])$ and so $\gamma$ must stabilize also $[x]$, \emph{i.e.}~$\Stab_\Gamma(y) \subseteq \Stab_\Gamma(x)$. As $\Stab_\Gamma(x)$ is of minimal cardinality these two must coincide, which proves that the set of points with minimal stabilizer is open.

We want to prove that the minimal stabilizer is trivial. Let $x \in V$ and $B \subseteq \IH^{n - 1} = \IP(V)$ be as above. The pre-image of $B$ in $V$ is then an open subset, containing a basis $e_1, \ldots, e_n$ for $V$. For $\gamma \in \Stab_\Gamma(x)$ we know from the previous part that $\gamma \in \Stab_\Gamma(e_i)$ for every $i$, which readily implies that $\gamma$ is the identity.
\end{proof}

\subsection{Cone conjectures}\label{sect: cone conjectures}
In this subsection, let $X$ be a \emph{projective} hyperk\"ahler manifold. Thanks to the BBF form (\S \ref{sect: action on BBF lattice}), the N\'eron--Severi space $\NS(X)_{\IR}$ is endowed with a non-degenerate symmetric bilinear form of signature $(1, \rho(X)-1)$. Let $\cA(X)\subseteq \NS(X)_{\IR}$ be the ample cone of $X$. Let $\Aut (\cA(X))$ be the group of isometries of $\NS(X)_\IR$ preserving the ample cone $\cA(X)$, the \emph{group of motions} of $\cA(X)$.

Recall that the natural action of $\Aut(X)$ on $\cA(X)$ is extended in \S \ref{subsect:Dagger} to an action of $\KAut(X)$, denoted by $\dagger$, see Lemma~\ref{lemma:dagger}.
Let $$\Aut^{*}(X)=\im(\Aut(X) \longrightarrow \Aut(\cA(X)))$$ and similarly for $\KAut^{*}(X)$.

For a hyperk\"ahler variety $X$, the Morrison--Kawamata cone conjecture is proved recently by Amerik--Verbitsky (\cite{AmerikVerbitsky}, \cite{AV-IMRN}) based on their earlier work \cite{AV-rational}, using hyperbolic geometry and ergodic theory. Their result says that $\Aut(X)$ \emph{acts with finitely many orbits on the set of facets of the K\"ahler cone of $X$} (see \cite[Theorem 2.13 and the discussion after]{AV-hyperbolic}). As a consequence, they deduce the second part of Theorem \ref{thm: m-k conj for three groups} below. 
Before that, Markman \cite{Markman} established a \emph{birational} analogue for the \emph{movable cone} $\MV(X)$, that is, the cone generated by the classes of movable divisors, and with the action on it given by the group \[\Bir^*(X) = \im(\Bir(X) \longrightarrow \Aut(\MV(X))).\]

The main result of this section is part (3) of the following theorem, which confirms the analogous cone conjecture for the action of Klein automorphisms $\KAut^{*}(X)$ on the rational closure (Definition \ref{def:RationalClosure}) $\cA^{+}(X)$ of the ample cone $\cA(X)$\,:
\begin{theorem}[Cone conjectures\,: extended]\label{thm: m-k conj for three groups}
Let $X$ be a projective hyperk\"ahler manifold. Then
\begin{enumerate}
\item ({\cite[Theorem 6.25]{Markman}}) There exists a rationally polyhedral cone $\Delta$ which is a fundamental domain for the action of $\Bir^*(X)$ on $\MV^{+}(X)$.
\item ({\cite[Theorem 5.6]{AmerikVerbitsky}}) There exists a rationally polyhedral cone $\Pi$ which is a fundamental domain for the action of $\Aut^*(X)$ on $\cA^{+}(X)$.
\item There exists a rationally polyhedral cone $\Sigma$ which is a fundamental domain for the action of $\KAut^*(X)$ on $\cA^{+}(X)$.
\end{enumerate}
\end{theorem}
\begin{proof}
$(2)$ The existence of a polyhedral fundamental domain for the action of $\Aut^*(X)$ is proved in \cite[\S 5.2]{AmerikVerbitsky} under the assumption that $b_2(X) \neq 5$, this last technical gap on Betti number was later filled in \cite[Corollary 1.4]{AV-IMRN}. The fact that there exists a \emph{rational} polyhedral fundamental domain then follows by applying Theorem \ref{thm:Looijenga} together with Proposition \ref{prop:TrivialStab}.


$(3)$ Note that both $\Aut^*(X)$ and $\KAut^*(X)$ preserve the ample cone $\cA(X)$ and the integral lattice $\NS(X)$ inside $\NS(X)_\IR$ (see \S \ref{subsect:Dagger}). Let $\Pi\subset \cA^{+}(X)$ be a polyhedral fundamental domain for the action of $\Aut^*(X)$ constructed in $(2)$. Then as $\KAut^*(X)$ contains $\Aut^*(X)$ (with finite index), one can apply Theorem \ref{thm:Looijenga} to $\Gamma=\KAut^{*}(X)$, to obtain the rationally polyhedral fundamental domain $\Sigma$. 
\end{proof}

\section{Proof of Theorem \ref{main2'} in the projective case}\label{sect:Proj}
Theorem~\ref{main2'} in the projective case takes the following form\,:
\begin{theorem}\label{main2-proj}
Let $X$ be a projective hyperk\"ahler variety. Then there are only finitely many conjugacy classes of finite subgroups of $\KAut(X)$, $\Aut(X)$ and $\Bir(X)$.
\end{theorem}
Recall that by Remark~\ref{rem:bridge} and Lemma~\ref{lemma:FiniteConj}, Theorem~\ref{main2-proj} implies the projective case of Theorem~\ref{main2}.

Let $X$ be a projective hyperk\"ahler variety throughout this section. To study the group $\KAut(X)$, let us break it into two pieces of different nature.
Consider the $\dagger$-action defined in \S \ref{subsect:Dagger} of $\KAut(X)$ on the N\'eron--Severi space $\NS(X)_{\IR}$ by isometries (Remark~\ref{rem:Isometry}). Let $\Aut(\cA(X))$ be the group of isometries of $\NS(X)_{\IR}$ preserving the ample cone $\cA(X)$. We have thus a homomorphism $\KAut(X)\to \Aut(\cA(X))$.
Denoting by $\KAut^{\#}(X)$ and $\KAut^{*}(X)$ its kernel and image respectively 
, we obtain a short exact sequence of groups\,:
\begin{equation}\label{eqn:BreakKAut}
 1\longrightarrow \KAut^{\#}(X)\longrightarrow\KAut(X)\longrightarrow\KAut^{*}(X)\longrightarrow 1.
\end{equation}
As $\cA(X)$ is open in $\NS(X)_{\IR}$, $\KAut^{\#}(X)$ is also the group of Klein automorphisms acting trivially on the N\'eron--Severi lattice $\NS(X)\simeq \Pic(X)$.

\begin{prop}\label{prop:KAutkernelfinite}
The group $\KAut^{\#}(X)$ is finite.
\end{prop}
\begin{proof}
Let $\Aut^{\#}(X)$ be the group of automorphisms of $X$ acting trivially on $\NS(X)$. As the product of any two non-trivial elements of $\KAut^{\#}(X)$ is in $\Aut^{\#}(X)$, $\Aut^{\#}(X)$ is of index at most $2$ in $\KAut^{\#}(X)$. Hence it is enough to show the finiteness of $\Aut^{\#}(X)$.

Let
\[\Aut^{s}(X) = \set{f \in \Aut(X) \st f^*|_{H^{2,0}(X)} = \id}\]
be the group of \emph{symplectic} automorphisms of $X$. As $X$ is projective, the transcendental lattice $T(X):=\NS(X)^{\perp_{BBF}}$ carries a polarizable irreducible Hodge structure (\emph{cf.}~\cite[Lemma~\textbf{3}.2.7 and Lemma~\textbf{3}.3.1]{HuyK3}). Hence $\Aut^{s}(X)$ also acts trivially on $T(X)$. Then the intersection $\Aut^{s}(X)\cap \Aut^{\#}(X)$ is the group of automorphisms acting trivially both on the transcendental and N\'eron--Severi lattices, hence trivially on the whole $H^{2}(X,\IZ)$ since $T(X)\oplus \NS(X)$ is of finite index in $H^{2}(X, \IZ)$, by the projectivity of $X$. Therefore, thanks to \cite[Proposition~9.1]{HuyInventiones}, $\Aut^{s}(X)\cap \Aut^{\#}(X)$ is a finite group.
 
On the other hand, by \cite[Corollary~\textbf{3}.3.4]{HuyK3}, $\Aut(X)/\Aut^{s}(X)$ is a finite cyclic group, hence so is its subgroup $\Aut^{\#}(X)/(\Aut^{s}(X)\cap \Aut^{\#}(X))$. In consequence, $\Aut^{\#}(X)$ is also finite.
\end{proof}

Concerning the action of $\Bir(X)$ on the movable cone $\MV(X)$, there is an exact sequence analogous to \eqref{eqn:BreakKAut}\,:
\begin{equation}\label{eqn:BreakBir}
1 \longrightarrow \Bir^{\#}(X) \longrightarrow \Bir(X)\longrightarrow \Bir^{*}(X)\longrightarrow 1,
\end{equation}
where $\Bir^*(X)$ and $\Bir^{\#}(X)$ are the image and the kernel of the natural homomorphism $\Bir(X) \longrightarrow \Aut(\MV(X))$.

\begin{cor}\label{cor:BirKerFinite}
The group $\Bir^{\#}(X)$ is finite.
\end{cor}
\begin{proof}
As the cone $\MV(X)$ is also open in $\NS(X)_{\IR}$, any $f \in \Bir^{\#}(X)$ acts trivially on $\NS(X)$. It then follows from Theorem \ref{thm:Markman} that $\Bir^{\#}(X) \subseteq \Aut(X) \cap \KAut^{\#}(X)$, which is then finite by Proposition \ref{prop:KAutkernelfinite}.
\end{proof}

Here comes the key point of the proof. It concerns $\KAut^{*}(X)$ and $\Bir^{*}(X)$, which are the images of the homomorphisms $\KAut(X) \longrightarrow \Aut(\cA(X))$ and $\Bir(X) \longrightarrow \Aut(\MV(X))$).

\begin{prop}\label{prop:key}
Let $X$ be a projective hyperk\"ahler manifold. Then there are only finitely many conjugacy classes of finite subgroups of $\KAut^{*}(X)$  and $\Bir^*(X)$.
\end{prop}
\begin{proof}
Let $G$ be a finite subgroup of $\KAut^{*}(X)$. Fix a rationally polyhedral fundamental domain $\Sigma$ for the action of $\KAut^{*}(X)$ on $\cA^{+}(X)$, whose existence is proved in Theorem \ref{thm: m-k conj for three groups}.  First of all, we observe that there exists a point $x \in \cA(X)$ such that $g.x = x$ for every $g \in G$. Indeed, $x = \sum_{g \in G} g.y$ for any point $y \in \cA(X)$ will work. Hence there exists $h \in \KAut(X)$ such that $x_0 = h^\dagger(x) \in \Sigma$. It follows that for every $g \in G$ we have
\[h^\dagger \circ g \circ (h^{-1})^\dagger(x_0) = h^\dagger \circ g\circ  (h^{-1})^\dagger\circ h^\dagger(x) = h^\dagger x = x_0,\]
\emph{i.e.}~the element $h^\dagger \circ g \circ (h^{-1})^\dagger$ fixes $x_0$ for every $g \in G$. This means that the subgroup $h^\dagger G {h^\dagger}^{-1}$ acts on $\cA(X)$ and fixes a point of $\Sigma$. Therefore
\[h^\dagger G {h^\dagger}^{-1} \subseteq \{ \varphi \in \KAut^*(X) \,|\, \varphi(\Sigma) \cap \Sigma \neq \{ 0 \} \}=:\cS.\]
We claim that $\cS$  is a finite set. By definition, for any $\varphi\in \cS$, $\varphi(\Sigma)$ and $\Sigma$ share at least a ray. On one hand, $\Sigma$ has only finitely many rays\,; and on the other hand, for each ray of $\Sigma$, there are only finitely many translates of $\Sigma$ by $\KAut^{*}(X)$ sharing it, thanks to the Siegel property (Proposition \ref{prop: Siegel}). Therefore  $\{ \varphi(\Sigma)\,|\, \varphi \in \cS\}$ is a finite set, which implies the finiteness of $\cS$ since $\Sigma$ is a fundamental domain.  In conclusion, any finite subgroup of $\KAut^{*}(X)$ can be conjugated to be contained in some given finite set $\cS$, which admits of course only finitely many subsets.

The proof for $\Bir^*(X)$ is exactly the same, provided we replace $\KAut^*(X)$, $\cA(X)$, $\cA^+(X)$ and $\Sigma$ by $\Bir^*(X)$, $\MV(X)$, $\MV^+(X)$ and $\Delta$ respectively.
\end{proof}

\begin{proof}[Proof of Theorem~\ref{main2-proj}]
We address first the group $\KAut(X)$. Recall that we have a short exact sequence \eqref{eqn:BreakKAut}\,:
\begin{equation*}
 1\longrightarrow \KAut^{\#}(X)\longrightarrow\KAut(X)\longrightarrow\KAut^{*}(X)\longrightarrow 1.
\end{equation*}
By Lemma~\ref{lemma:FiniteConj} and Proposition~\ref{prop:key}, we have the following two facts\,:
\begin{enumerate}
\item[(1*)] The cardinalities of finite subgroups of $\KAut^{*}(X)$ are bounded.
\item[(2*)] For any finite group $G$, the set $H^{1}(G, \KAut^{*}(X))$ is finite, where $G$ acts on $\KAut^{*}(X)$ trivially.
\end{enumerate}
Again by Lemma~\ref{lemma:FiniteConj}, it is enough to establish these two properties for $\KAut(X)$\,: 
\begin{enumerate}
\item The cardinalities of finite subgroups of $\KAut(X)$ are bounded.
\item For any finite group $G$, the set $H^{1}(G, \KAut(X))$ is finite, where $G$ acts on $\KAut(X)$ trivially.
\end{enumerate}
For (1), let $G$ be any finite subgroup of $\KAut(X)$, then $G\cap \KAut^{\#}(X)$ is a finite group of bounded cardinality by Proposition~\ref{prop:KAutkernelfinite} and $G/G\cap \KAut^{\#}(X)$ is a finite subgroup of $\KAut^{*}(X)$, hence has bounded cardinality by (1*). Therefore, the cardinality of $G$ is bounded.\\
For (2), fix any finite group $G$, the short exact sequence \eqref{eqn:BreakKAut}, with trivial $G$-actions, induces an exact sequence of pointed sets (where the first map is injective by using Remark \ref{rmk:fiberbis}, but we do not need this here)\,:
\begin{equation*}
H^{1}(G, \KAut^{\#}(X))\longrightarrow H^{1}(G, \KAut(X))\longrightarrow H^{1}(G, \KAut^{*}(X)).
\end{equation*}
The third term being finite (thanks to (2*)), the finiteness of the middle term is equivalent to the finiteness of the fibers of the second map, which by Lemma~\ref{lemma:twist} is implied by the finiteness of the cohomology sets $H^{1}(G, \KAut^{\#}(X)_{\phi})$ for all $\phi\in Z^{1}(G, \KAut(X))$, where $\KAut^{\#}(X)_{\phi}$ is the group $\KAut^{\#}(X)$ endowed with a $G$-action twisted by the $1$-cocycle $\phi$. As $\KAut^{\#}(X)_{\phi}$ is in any way a finite $G$-group, $H^{1}(G, \KAut^{\#}(X)_{\phi})$ is obviously finite by definition. The proof is therefore complete.

As $\Aut(X)$ has finite index in $\KAut(X)$, the result for $\Aut(X)$ follows from Lemma \ref{lemma:FiniteConj}.

Finally, the proof for $\Bir(X)$ is the exactly same as the one for $\KAut(X)$ by replacing $\KAut^{\#}(X)$, $\KAut^{*}(X)$ and Proposition \ref{prop:KAutkernelfinite} by $\Bir^{\#}(X)$, $\Bir^{*}(X)$ and Corollary \ref{cor:BirKerFinite} respectively. 
\end{proof}

\begin{rem}
The proof of the key Proposition \ref{prop:key} actually provides a bound for the orders of finite subgroups of $\KAut^{*}(X)$ (resp.~$\Bir^*(X)$), namely $|\cS|$, the number of translates of the fundamental domain $\Sigma$ (resp.~$\Delta$) that share at least a ray with $\Sigma$ (resp.~$\Delta$). This would lead a bound for the orders of finite subgroups of $\KAut(X)$ and $\Bir(X)$. Let $G \leq \KAut(X)$ be a finite subgroup, and call $G^\#$ the subgroup of $G$ consisting of those Klein automorphisms acting trivially on $\NS(X)$ and $G^*$ the image of $G$ in $O(\NS(X))$. Then we have
\[|G| = |G^\#| |G^*| \leq |\cS| \cdot 2 |G^\# \cap \Aut(X)|,\]
so to find a bound for $|G|$ we just need to bound $|G^\# \cap \Aut(X)|$. Looking now at the action of elements in this group on the transcendental lattice $T(X)$ and arguing as in the proof of Proposition \ref{prop:KAutkernelfinite}, we see that we can bound the cardinality of the subgroup $\{g \in G^\# \cap \Aut(X) \,|\, g^*|_{H^{2, 0}(X)} = \id \}$ by the cardinality of the group $\{ g \in \Aut(X) \,|\, g^*|_{H^2(X, \IZ)} = \id \}$, which as we mentioned is finite and depends only on the deformation type of $X$. On the other hand, the quotient group naturally embeds in $\Aut(X) / \Aut^s(X)$, which is a finite cyclic group of order say $m$. It is known by \cite[Proposition 7]{Beauville2} that $\varphi(m) \leq b_2(X) - \rho(X) \leq b_2(X) - 1$, hence also $m$ can be bounded by a constant which depends only on the deformation type of $X$.
\end{rem}

\section{Proof of Theorem \ref{main2'} in the non-projective case}\label{sect:Non-proj}
The non-projective case of Theorem \ref{main2'} is the following\,:
\begin{theorem}\label{main2-non-proj}
Let $X$ be a non-projective compact hyperk\"ahler manifold. Then there are only finitely many conjugacy classes of finite subgroups of $\KAut(X)$, $\Aut(X)$ and $\Bir(X)$.
\end{theorem}
Recall that by Remark~\ref{rem:bridge} and Lemma~\ref{lemma:FiniteConj}, the non-projective case of Theorem~\ref{main2} is a consequence of Theorem \ref{main2-non-proj}.

Although the BBF lattice $H^{2}(X, \IZ)$ is non-degenerate of signature $(3, b_{2}(X)-3)$ (\emph{cf.}~\S \ref{sect: action on BBF lattice}), its restriction to the N\'eron--Severi lattice $\NS(X)$ has three possibilities in general (\emph{cf.}~\cite{Oguiso})\,:
\begin{enumerate}
\item a \emph{hyperbolic} lattice of signature $(1, 0, \rho - 1)$,
\item an \emph{elliptic} lattice of signature $(0, 0, \rho)$,
\item a \emph{parabolic} lattice of signature $(0, 1, \rho- 1)$,
\end{enumerate}
where $\rho=\rho(X)$ is the Picard rank of $X$. It is a theorem of Huybrechts \cite[Thm.~3.11]{HuyInventiones} and \cite{HuyErratum} that the projectivity of $X$ is equivalent to the first case that $\NS(X)$ is hyperbolic.

Let $X$ be a non-projective compact hyperk\"ahler manifold in the sequel of this section. Hence $\NS(X)$ with the restriction of the BBF form $q$, is either elliptic or parabolic. Let 
$$R:=\ker(q|_{\NS(X)})$$ be the \emph{radical} of $\NS(X)$, which is either trivial or isomorphic to $\IZ$.

\begin{lemma}\label{lemma:TNSvsH2}
Let $f$ be a birational automorphism of a compact hyperk\"ahler variety $X$. If $f$ acts trivially on $\NS(X)$ and $H^{2,0}(X)$, then it acts trivially on the whole $H^{2}(X, \IZ)$.
\end{lemma}
\begin{proof}
The assumptions imply that $f$ acts trivially on the transcendental lattice $T(X):=\NS(X)^{\perp}$ (\emph{cf.}~\cite[Lemma~\textbf{3}.3.3 and the discussion after that]{HuyK3}). If $\NS(X)$ is (hyperbolic or) elliptic, as $T(X)\oplus\NS(X)$ has finite index in $H^{2}(X,\IZ)$, $f^{*}$ acts trivially on $H^{2}(X)$. 

If $\NS(X)$ is parabolic with radical $R$, then $\NS(X)_{\IQ}\cap T(X)_{\IQ}=R_{\IQ}$ and $\NS(X)_{\IQ}+ T(X)_{\IQ}$ is of codimension 1 in $H^{2}(X,\IQ)$. 
In this case, let $N = {(\NS X + T(X))^\perp}^\perp$ be the saturation; this is a primitive sublattice of $H^2(X, \IZ)$, having $\NS X + T(X)$ as finite index sublattice, and the quadratic form on $H^2(X, \IZ)$ restricts to a degenerate form on $N$ with kernel of rank $1$ generated by an element $v$. We can then find a $\IZ$-basis of $H^2(X, \IZ)$ of the form $e_1, \ldots, e_t, v, h$ with $e_1, \ldots, e_t, v$ a $\IZ$-basis for $N$.  With respect to this basis, the Gram matrix of the Beauville--Bogomolov--Fujiki form is
\[G = \left( \begin{array}{c|c|c}
G_0 & 0 & G_1\\
\hline
0 & 0 & v \cdot h\\
\hline
G_1^T & v \cdot h & h^2
\end{array} \right),\]
where $G_0$ is a $t \times t$ matrix, $G_1 \in \IZ^t$ and $G_1^T$ is its transpose. Observe that $\det G_0 \neq 0$ as $N$ has a rank $1$ kernel generated by $v$, and that $v \cdot h \neq 0$ for otherwise the pairing on $H^2(X, \IZ)$ would be degenerate.

Let now $\varphi$ be an isometry of $H^2(X, \IZ)$ which restricts to the identity on $\NS X$ and on $T(X)$, or equivalently which restricts to the identity on $N$. The matrix representing $\varphi$ with respect to our basis is of the form
\[M = \left( \begin{array}{c|c|c}
\id_t & 0 & x\\
\hline
0 & 1 & y\\
\hline
0 & 0 & z
\end{array} \right),\]
where $x \in \IZ^t$ and $y, z \in \IZ$. As this matrix is invertible, we have that $z = \pm 1$, and as it represents an isometry we must have $M^T G M = G$. A computation gives that
\[M^T G M = \left( \begin{array}{c|c|c}
G_0 & 0 & G_0 x \pm G_1\\
\hline
0 & 0 & \pm v \cdot h\\
\hline
x^T G_0 \pm G_1^T & \pm v \cdot h & x^T G_0 x \pm 2 G_1^T x \pm 2y v \cdot h + h^2
\end{array} \right),\]
where the sign $\pm$ refers to $z = \pm 1$. Since this matrix should equal $G$, from the fact that $v \cdot h \neq 0$ we deduce that $z = 1$. Hence, looking at the top right block we have $G_0 x + G_1 = G_1$, from which $x = 0$ as $\det G_0 \neq 0$. Finally, the bottom right block reads $2y v \cdot h + h^2$, and since this must equal $h^2$ and $v \cdot h \neq 0$, we have $y = 0$. But then $M$ is the matrix of the identity and we are done.
\end{proof}

\begin{proof}[Proof of Theorem~\ref{main2-non-proj}]
Notation is as before. 
The strategy is to apply Lemma~\ref{lemma:filtration} to natural filtrations of $\Aut(X)$ and $\Bir(X)$. In the following proof, let $A$ denote either $\Aut(X)$ or $\Bir(X)$.
Consider the following normal subgroups of $A$\,:
\begin{itemize}
\item $A_{1}:= \left\{f\in A ~\mid~ f^{*}|_{R}=\id\right\}$\,;
\item $A_{2}:= \left\{f\in A ~\mid~ f^{*}|_{R}=\id; f^{*}|_{\NS(X)/R}=\id\right\}$\,;
\item $A_{3}:= \left\{f\in A ~\mid~ f^{*}|_{\NS(X)}=\id\right\}$\,;
\item $A_{4}:= \left\{f\in A ~\mid~ f^{*}|_{\NS(X)}=\id; f^{*}|_{H^{2,0}(X)}=\id\right\}$\,;
\item $A_{5}:= \left\{f\in A ~\mid~ f^{*}|_{H^{2}(X)}=\id\right\}$,
\end{itemize}
which form a filtration\,:
$$1 \subseteq A_{5}\subseteq A_{4}\subseteq A_{3}\subseteq A_{2}\subseteq A_{1}\subseteq A.$$
Let us verify that the successive graded subquotients of this filtration satisfy the hypotheses of Lemma~\ref{lemma:filtration}, \emph{i.e.}~being finite or abelian of finite type\,:
\begin{itemize}
\item $A/A_{1}$ is a subgroup of $\Aut(R)$, which is either $\set{\pm 1}$ when $R$ is of rank $1$, or zero when $R$ is trivial. In any case, it is finite.
\item $A_{1}/A_{2}$ is by construction isomorphic to a subgroup of the automorphism group of the elliptic (\emph{i.e.}~negative definite) lattice $\NS(X)/R$, which is obviously a finite group.
\item $A_{2}/A_{3}$ is by construction isomorphic to a subgroup of $\Hom_{\IZ}(\NS(X)/R, R)$ which is a free abelian group of finite rank (possibly zero).
\item $A_{3}/A_{4}$ is by construction isomorphic to a subgroup of the image of $$\Bir(X)\to \GL(H^{2,0}(X))\simeq \IC^{*},$$ which is either $\IZ$ or trivial by Oguiso \cite[Theorem~2.4, Propositions~4.3, 4.4]{Oguiso}.
\item Finally, we have $A_{4}=A_{5}$ by Lemma~\ref{lemma:TNSvsH2}, and this last is finite by \cite[Proposition~9.1]{HuyInventiones}.
\end{itemize}
Therefore, we see that all graded pieces of the filtration are either finite or abelian of finite type, one can conclude for $\Aut(X)$ and $\Bir(X)$ by Lemma~\ref{lemma:filtration}. As for $\KAut(X)$, the above filtration for $A=\Aut(X)$ consists of normal subgroups of $\KAut(X)$ and $\KAut(X)/\Aut(X)$ is at most of order $2$, so Lemma~\ref{lemma:filtration} applies.
\end{proof}

\begin{rem}
Oguiso \cite{Oguiso} shows that for a non-projective compact hyperk\"ahler manifold $X$, its bimeromorphic automorphism group $\Bir(X)$ is \emph{almost abelian of finite rank}, that is, isomorphic to a finite-rank free abelian group, up to finite kernel and cokernel (see \cite[\S 8]{Oguiso}). Hence the same holds for $\Aut(X)$ and $\KAut(X)$. Unfortunately, we are not able to deduce our finiteness Theorem~\ref{main2-non-proj} from this very strong result. The issue is related to Remark~\ref{rm:NormalityPb} about the normality hypothesis in the filtration Lemma~\ref{lemma:filtration}. However, the authors believe that the subgroups appeared in the proof of Oguiso's theorem are indeed normal in $\KAut(X)$ and $\Bir(X)$. Moreover, if one is only interested in the finiteness of real structures, then Oguiso's theorem is enough\,: if a group $A$ is almost abelian of finite rank, then $H^{1}(\IZ/2\IZ, A)$ is finite, where $\IZ/2\IZ$ acts trivially on $A$.
\end{rem}

\section{Finiteness of real structures: proof of Theorem \ref{main1}}\label{sect:FiniteReal}
For a compact hyperk\"ahler manifold, assume that there exists at least one real structure (Definition~\ref{def:RealStr}). In this case, we have a splitting short exact sequence
\begin{equation}\label{eqn:sSES}
\xymatrix{
1\ar[r] &\Aut(X)\ar[r] &\KAut(X)\ar[r]^{\epsilon} & \set{\pm 1} \ar@/^1pc/[l] \ar[r] &1
}
\end{equation}
\begin{proof}[Proof of Theorem \ref{main1}]
Let us fix a real structure $\sigma$. By Lemma~\ref{lemma:CohomForRealStr}, we need to show the finiteness of the cohomology set $H^{1}(\IZ/2\IZ, \Aut(X))$, where $\Aut(X)$ is endowed with the action of conjugation by $\sigma$. The short exact sequence \eqref{eqn:sSES} induces an exact sequence of pointed sets\,:
\[\cdots\to \set{\pm 1}\to H^{1}(\IZ/2\IZ, \Aut(X))\to H^{1}(\IZ/2\IZ, \KAut(X))\to \cdots.\]
With $\set{\pm 1}$ being finite, it suffices, by \cite[Corollaire~1.13]{BorelSerre}, to show that the cohomology set  $H^{1}(\IZ/2\IZ, \KAut(X))$ is finite. However the action of $\IZ/2\IZ$ on $\KAut(X)$ is given by the conjugation by $\sigma$, \emph{i.e.}~an inner automorphism, so by \cite[Proposition~1.5]{BorelSerre}, $H^{1}(\IZ/2\IZ, \KAut(X))$ is in bijection with $H^{1}(\IZ/2\IZ, \KAut(X)_{\triv})$ where $\KAut(X)_{\triv}$ is endowed with the trivial $\IZ/2\IZ$-action. Finally, the complement of the base point (the trivial cocycle) in $H^{1}(\IZ/2\IZ, \KAut(X)_{\triv})$ is naturally identified with the set of conjugacy classes of subgroups of order $2$ in $\KAut(X)$, thus its finiteness is a special case of Theorem~\ref{main2}.
\end{proof}

\section{Finiteness properties of automorphism groups: proof of Theorem \ref{main3}}\label{sect:FiniteGeneration}

The goal of this section is to show some strong finiteness properties of the automorphism group and the birational automorphism group of a compact hyperk\"ahler variety (Theorem~\ref{main3}), namely finite presentation and $(\FP_{\infty})$ property.

Let us briefly recall various finiteness properties involved. Some standard references are \cite{Brown} and \cite{MR715779}.

\begin{definition}[Finiteness properties of groups \cite{Brown}]
Let $\Gamma$ be a group.
\begin{enumerate}
\item $\Gamma$ is called \emph{of type} $(\FL)$ (\emph{resp.~of length} $\leq n$) if the trivial $\IZ[\Gamma]$-module $\IZ$ has a finite resolution (\emph{resp.}~of length $n$) by free $\IZ[\Gamma]$-modules of finite rank\,:
\[\xymatrix{0 \ar[r] & \IZ[\Gamma]^{m_n} \ar[r] & \cdots \ar[r] & \IZ[\Gamma]^{m_1} \ar[r] &  \IZ[\Gamma]^{m_0} \ar[r] &  \IZ \ar[r] &  0}.\]
\item $\Gamma$ is said to be \emph{of type} $(\FP)$ (\emph{resp.~of length} $\leq n$) if the trivial $\IZ[\Gamma]$-module $\IZ$ admits a finite resolution (\emph{resp.}~of length $n$) by finitely generated projective $\IZ[\Gamma]$-modules\,:
\[\xymatrix{0 \ar[r] & P_n \ar[r] & \cdots \ar[r] & P_1 \ar[r] &  P_0 \ar[r] &  \IZ \ar[r] &  0}.\]
\item Let $n\in \IN$, we say that $\Gamma$ is \emph{of type} $(\FP_n)$ if the trivial $\IZ[\Gamma]$-module $\IZ$ has a length-$n$ partial resolution by finitely generated projective $\IZ[\Gamma]$-modules\,:
\[\xymatrix{P_n \ar[r] & \ldots \ar[r] & P_1 \ar[r] &  P_0 \ar[r] &  \IZ \ar[r] &  0}.\]
We say $\Gamma$ is \emph{of type} $(\FP_\infty)$ if it is \emph{of type} $(\FP_n)$ for all $n \geq 0$.
\item We say $\Gamma$ \emph{virtually} satisfies a property if it admits a finite-index subgroup satisfying this property. We can therefore define properties like \emph{virtual} $(\FL)$ and \emph{virtual} $(\FP)$, denoted by $(\VFL)$ and $(\VFP)$ respectively.
\end{enumerate}
 It follows from definitions that $\Gamma$ is of type $(\FP)$ if and only if $\Gamma$ is of type $(\FP_{\infty})$ and the ring $\IZ[\Gamma]$ is of finite cohomological dimension (\cite[Chapter VIII, Proposition 6.1]{Brown}). 
 For any $0\leq n\leq \infty$, a group $\Gamma$ is of type $(\FP_{n})$ is equivalent to the same condition for any finite-index subgroup (\cite[Chapter VIII, Proposition 5.1]{Brown}). Hence the ``virtual $(\FP_{n})$ property'' coincides with $(\FP_{n})$ itself and $(\VFP)$ implies $(\FP_{\infty})$.
\end{definition}

The following diagram in Figure \ref{eqn:Implications} summarizes some known implications  (\emph{cf.}~\cite[Chapter VIII]{Brown})\,:
\begin{figure}[H]
\begin{tikzpicture}
\node[draw,rectangle,text centered, rounded corners,text width=3em,minimum height=3em] (FL){$\FL$};
\node[draw,rectangle,text centered, rounded corners,text width=3em,minimum height=3em, below=1cm of FL] (VFL){$\VFL$};
\node[draw,rectangle,text centered, rounded corners,text width=3em,minimum height=3em, right=1cm of FL] (FP){$\FP$};
\node[draw,rectangle,text centered, rounded corners,text width=3em,minimum height=3em, below=1cm of FP] (VFP){$\VFP$};
\node[draw,rectangle,text centered, rounded corners,text width=3em,minimum height=3em, left=1cm of VFL] (Finite){Finite};
\node[draw,rectangle,text centered, rounded corners,text width=3em,minimum height=3em, right=1cm of VFP] (FPINFTY){$\FP_{\infty}$};
\node[draw=none, fill=none, text centered, text width=0.7em,minimum height=2em, below=0.5cm of FPINFTY] (blabla1){...};
\node[draw,rectangle,text centered, rounded corners,text width=3em,minimum height=3em, below=0.5cm of blabla1] (FPn){$\FP_{n}$};
\node[draw=none, fill=none, text centered, text width=0.7em,minimum height=2em, below=0.5cm of FPn] (blabla2){...};
\node[draw,rectangle,text centered, rounded corners,text width=3em,minimum height=3em, below=0.5cm of blabla2] (FP2){$\FP_{2}$};
\node[draw,rectangle,text centered, rounded corners,text width=3em,minimum height=3em, below=1cm of FP2] (FP1){$\FP_{1}$};
\node[draw,rectangle,text centered, rounded corners,text width=6em,minimum height=3em, right=1cm of FP2] (FPre){Finite presentation};
\node[draw,rectangle,text centered, rounded corners,text width=6em,minimum height=3em, right=1cm of FP1] (FG){Finite generation};
\draw[-latex',double](FL)--(VFL);
\draw[-latex',double](FL)--(FP);
\draw[-latex',double](FP)--(VFP);
\draw[-latex',double](VFL)--(VFP);
\draw[-latex',double](Finite)--(VFL);
\draw[-latex',double](VFP)--(FPINFTY);
\draw[-latex',double](FPINFTY)--(blabla1);
\draw[-latex',double](blabla1)--(FPn);
\draw[-latex',double](FPn)--(blabla2);
\draw[-latex',double](blabla2)--(FP2);
\draw[-latex',double](FP2)--(FP1);
\draw[-latex',double](FPre)--(FP2);
\draw[latex'-latex',double](FP1)--(FG);
\draw[-latex',double](FPre)--(FG);
\draw[-latex',double](FPINFTY)--node[right]{$cd< \infty$}(FP);
\end{tikzpicture}
\caption{Finiteness properties of groups}
\label{eqn:Implications}
\end{figure}
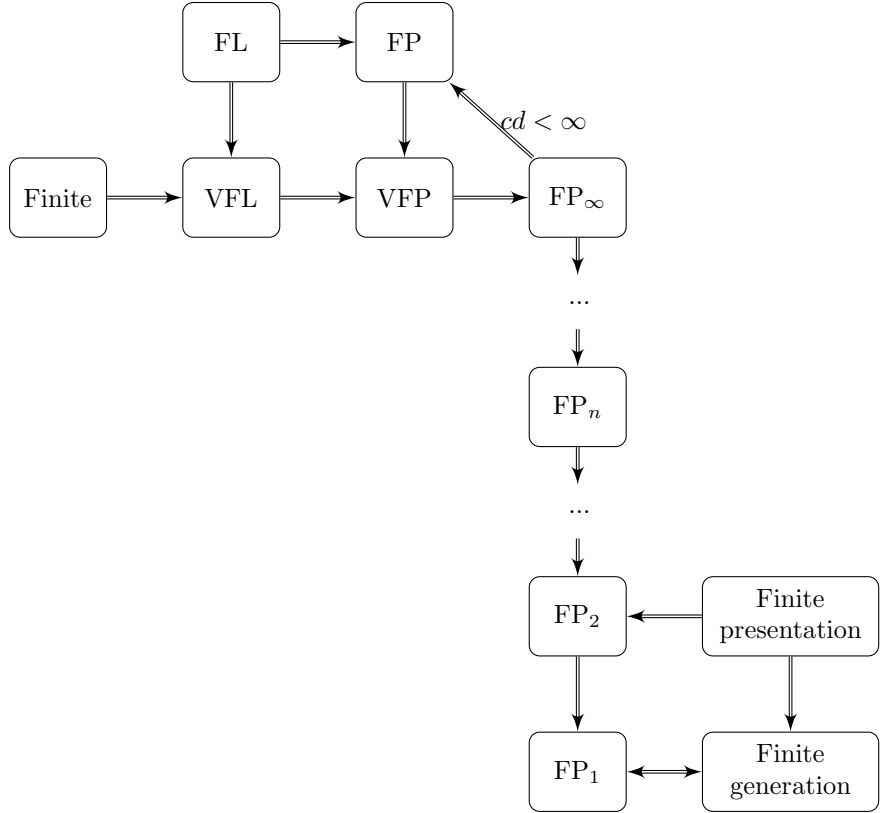

We will need the following fact\,:
\begin{prop}[Extensions {\cite[Proposition 2.7]{MR715779}}, {\cite[\S 10, Corollary 2]{MR1472735}}]\label{prop:Extension}
Let $\Gamma'$ be a normal subgroup of a group $\Gamma$ with quotient group $\Gamma''$\,:
$$0\to \Gamma'\to \Gamma\to \Gamma''\to 0.$$
\begin{itemize}
\item  Assume that $\Gamma'$ is of type $(\FP_{\infty})$. Then for any $n\in \IN\cup \{\infty\}$, $\Gamma$ is of type $(\FP_{n})$ if and only if $\Gamma''$ is so. In particular, the $(\FP_{\infty})$ property is preserved under extensions.
\item If $\Gamma'$ and $\Gamma''$ are both finitely presented, then so is $\Gamma$.
\end{itemize}
\end{prop}
\begin{rem}
Actually, all the finiteness properties in Figure \ref{eqn:Implications} are all preserved under extensions (\cite[Chapter VIII, \S 6, Exercise 8]{Brown}, \cite[P.23, Exercise]{MR715779}), except for $(\VFL)$ and $(\VFP)$, where one has to require moreover the condition of virtual torsion-freeness (\cite[Chapter VIII, \S 11, Exercise 2]{Brown}).
\end{rem}

\bigskip

Now let us return to the automorphism group and the birational automorphism group of a compact hyperk\"ahler manifold. As we explained in $\S$\ref{sect: finite presentation intro}, in the non-projective case the result of Oguiso \cite{Oguiso} says that $\Bir(X)$ and $\Aut(X)$ are both almost abelian of finite rank, in particular finitely presented and of type $(\FP_{\infty})$. Therefore we restrict ourselves in the sequel to the projective case. 

The key ingredient in our proof is the following result on convex geometry which is due to Looijenga.

\begin{prop}[{\cite[Corollaries 4.15 and 4.16]{Looijenga}}]\label{prop: fin presented}
Let $C$ be a non-degenerate open convex cone in a finite dimensional real vector space $V$ equipped with a $\IQ$-structure. Let $\Gamma$ be a subgroup of $\GL(V)$ which preserves $C$ and some lattice in $V(\IQ)$. If there exists a polyhedral cone $\Pi$ in $C^+$ such that $\Gamma \cdot \Pi \supseteq C$, then $\Gamma$ is finitely presented and of type $(\VFL)$ of length $\leq \dim(V)-1$.
\end{prop}

Now we have all the ingredients to show our finiteness result\,:
\begin{proof}[Proof of Theorem~\ref{main3}]
Let $X$ be a projective hyperk\"ahler manifold. We treat firstly its automorphism group. Consider the exact sequence
\[\xymatrix{1 \ar[r] & \Aut^{\#}(X) \ar[r] & \Aut(X) \ar[r] & \Aut^*(X) \ar[r] & 1},\]
where $\Aut^{\#}(X)$ and $\Aut^{*}(X)$ are respectively the kernel and the image of the natural representation $\Aut(X)\to O(\NS(X))$.
On one hand, the existence of a polyhedral fundamental domain for the action of $\Aut^{*}(X)$ on the rational closure of the ample cone (Theorem \ref{thm: m-k conj for three groups}) allows us to apply Proposition \ref{prop: fin presented} and conclude that $\Aut^*(X)$ is finitely presented and of type $(\VFL)$. On the other hand, by Proposition \ref{prop:KAutkernelfinite}, $\Aut^{\#}(X)$ is a finite group, which is of course finitely presented and of type $(\VFL)$. As a result, $\Aut(X)$ is an extension of two finitely presented groups of type $(\VFL)$,  hence in particular of type $(\FP_{\infty})$, see Figure \ref{eqn:Implications}. By Proposition \ref{prop:Extension}, $\Aut(X)$ is also finitely presented and of type $(\FP_{\infty})$.

The above argument applies equally to the birational automorphism group $\Bir(X)$. Indeed, Markman \cite{Markman} shows that the action of $\Bir(X)$, or rather its image $\Bir^{*}(X)$ under the map of restriction to the N\'eron--Severi space,  on the rational closure of the movable cone has a rational polyhedral fundamental domain and Looijenga's result Proposition \ref{prop: fin presented} implies that $\Bir^{*}(X)$ is finitely presented and of type $(\VFL)$. We still have the finiteness of $\Bir^{\#}(X)=\ker\left(\Bir(X)\to \Bir^{*}(X)\right)$ (Corollary \ref{cor:BirKerFinite}) and so one can conclude as in the case for $\Aut(X)$ using Proposition \ref{prop:Extension}.
\end{proof}

\begin{rem}[Klein automorphism group]
 As $\Aut(X)$ is normal and of finite index in $\KAut(X)$, this last is also finitely presented and of type $(\FP_{\infty})$ by Proposition \ref{prop:Extension}.
\end{rem}


\begin{rem}[$\Bir$ \emph{vs.}~$\Aut$]
It is asked in \cite[Question 1.6]{MR2231119} whether, for a projective hyperk\"ahler variety $X$, the index of $\Aut(X)$ inside $\Bir(X)$ is always finite or not. The answer to this question is negative in general. The first counter-example is constructed by Hassett--Tschinkel \cite[Theorem 7.4, Remark 7.5]{MR2576800} (where $\Aut(X)$ is trivial while $\Bir(X)$ is infinite) using Fano varieties of lines of special cubic fourfolds; then Oguiso gives a systematic study in the Picard rank two case \cite[Theorem 1.3]{MR3263669}. We thank Ekaterina Amerik for the references.
\end{rem}

\begin{rem}[Arithmeticity]
It is in general not true that for a projective hyperk\"ahler variety $X$, the groups $\KAut(X)$ and $\Aut(X)$ are arithmetic groups. Counter-examples exist already for K3 surfaces\,: the first one is due to Borcherds \cite[Example 5.8]{MR1654763}\,; and later Totaro \cite[\S 6]{MR2931877} provided explicit examples of K3 surfaces whose automorphism group is not even commensurable with an arithmetic group. If $\KAut(X)$ or $\Aut(X)$ were arithmetic, our main results Theorems~\ref{main1}, \ref{main2} and \ref{main3} would be direct consequences. This line of consideration raises the following natural question\,: \emph{how is $\Aut(X)$ related to arithmetic groups\,?} For projective K3 surfaces, the subgroup of symplectic automorphisms $\Aut^{s}(X)$ is itself arithmetic (\emph{cf.} \cite[Corollary \textbf{14}.2.4]{HuyK3}); while for higher dimensional projective hyperk\"ahler manifolds, it is known that the group $\Bir(X)$ is, up to a finite kernel and cokernel, a quotient of an arithmetic group by a reflection group (this is Boissi\`ere--Sarti's proof of the finite generation of $\Bir(X)$, see \cite[Theorem 2]{MR2957195}). One could ask whether the automorphism group $\Aut(X)$ is \emph{almost} arithmetic, \emph{i.e.}~it is arithmetic up to a finite kernel and finite cokernel.
\end{rem}

\bibliographystyle{alpha}
\bibliography{Klein_actions_and_real_structures}

\begin{thebibliography}{DFMJ18}

\bibitem[AV15]{AV-rational}
E.~Amerik and M.~Verbitsky.
\newblock Rational curves on hyperk\"ahler manifolds.
\newblock {\em Int. Math. Res. Not. IMRN}, (23):13009--13045, 2015.

\bibitem[AV16]{AV-hyperbolic}
Ekaterina Amerik and Misha Verbitsky.
\newblock Hyperbolic geometry of the ample cone of a hyperk\"ahler manifold.
\newblock {\em Res. Math. Sci.}, 3:Paper No. 7, 9, 2016.

\bibitem[AV17]{AmerikVerbitsky}
Ekaterina Amerik and Misha Verbitsky.
\newblock Morrison-{K}awamata cone conjecture for hyperk\"ahler manifolds.
\newblock {\em Ann. Sci. \'Ec. Norm. Sup\'er. (4)}, 50(4):973--993, 2017.

\bibitem[AV18]{AV-IMRN}
Ekaterina Amerik and Misha Verbitsky.
\newblock {C}ollections of orbits of hyperplane type in homogeneous spaces,
  homogeneous dynamics, and hyperk{\"a}hler geometry.
\newblock {\em International Mathematics Research Notices}, page rnx319, 2018.

\bibitem[BD85]{BeauvilleDonagi}
Arnaud Beauville and Ron Donagi.
\newblock La vari\'et\'e des droites d'une hypersurface cubique de dimension
  {$4$}.
\newblock {\em C. R. Acad. Sci. Paris S\'er. I Math.}, 301(14):703--706, 1985.

\bibitem[Bea83a]{Beauville2}
Arnaud Beauville.
\newblock Some remarks on k{\"a}hler manifolds with c1= 0.
\newblock {\em Classification of algebraic and analytic manifolds (Katata,
  1982)}, 39:1--26, 1983.

\bibitem[Bea83b]{Beauville}
Arnaud Beauville.
\newblock Vari\'et\'es {K}\"ahleriennes dont la premi\`ere classe de {C}hern
  est nulle.
\newblock {\em J. Differential Geom.}, 18(4):755--782 (1984), 1983.

\bibitem[Ben16]{Benzerga}
Mohamed Benzerga.
\newblock Real structures on rational surfaces and automorphisms acting
  trivially on {P}icard groups.
\newblock {\em Math. Z.}, 282(3-4):1127--1136, 2016.

\bibitem[Bie81]{MR715779}
Robert Bieri.
\newblock {\em Homological dimension of discrete groups}.
\newblock Queen Mary College Mathematical Notes. Queen Mary College, Department
  of Pure Mathematics, London, second edition, 1981.

\bibitem[Bog74]{Bogomolov}
F.~A. Bogomolov.
\newblock The decomposition of {K}\"ahler manifolds with a trivial canonical
  class.
\newblock {\em Mat. Sb. (N.S.)}, 93(135):573--575, 630, 1974.

\bibitem[Bog78]{MR514769}
F.~A. Bogomolov.
\newblock Hamiltonian {K}\"ahlerian manifolds.
\newblock {\em Dokl. Akad. Nauk SSSR}, 243(5):1101--1104, 1978.

\bibitem[Bor98]{MR1654763}
Richard~E. Borcherds.
\newblock Coxeter groups, {L}orentzian lattices, and {$K3$} surfaces.
\newblock {\em Internat. Math. Res. Notices}, (19):1011--1031, 1998.

\bibitem[Bro82]{Brown}
Kenneth~S. Brown.
\newblock {\em Cohomology of groups}, volume~87 of {\em Graduate Texts in
  Mathematics}.
\newblock Springer-Verlag, New York-Berlin, 1982.

\bibitem[BS64]{BorelSerre}
A.~Borel and J.-P. Serre.
\newblock Th\'eor\`emes de finitude en cohomologie galoisienne.
\newblock {\em Comment. Math. Helv.}, 39:111--164, 1964.

\bibitem[BS12]{MR2957195}
Samuel Boissi\`ere and Alessandra Sarti.
\newblock A note on automorphisms and birational transformations of holomorphic
  symplectic manifolds.
\newblock {\em Proc. Amer. Math. Soc.}, 140(12):4053--4062, 2012.

\bibitem[DFMJ18]{DFMJ}
Adrien Dubouloz, Gene Freudenburg, and Lucy Moser-Jauslin.
\newblock Algebraic vector bundles on the 2-sphere and smooth rational
  varieties with infinitely many real forms.
\newblock {\em Preprint, arXiv: 1807.05885.}, 2018.

\bibitem[DIK00]{RealEnriques}
A.~Degtyarev, I.~Itenberg, and V.~Kharlamov.
\newblock {\em Real {E}nriques surfaces}, volume 1746 of {\em Lecture Notes in
  Mathematics}.
\newblock Springer-Verlag, Berlin, 2000.

\bibitem[Dil11]{Diller11}
Jeffrey Diller.
\newblock Cremona transformations, surface automorphisms, and plane cubics.
\newblock {\em Michigan Math. J.}, 60(2):409--440, 2011.
\newblock With an appendix by Igor Dolgachev.

\bibitem[DO18]{DinhOguiso}
Tien-Cuong Dinh and Keiji Oguiso.
\newblock A surface with a discrete and nonfinitely generated automorphism
  group.
\newblock {\em Preprint, arXiv: 1710.07019. To appear in Duke Mathematical
  Journal.}, 2018.

\bibitem[DV10]{DebarreVoisin}
Olivier Debarre and Claire Voisin.
\newblock Hyper-k{\"a}hler fourfolds and grassmann geometry.
\newblock {\em Journal f{\"u}r die reine und angewandte Mathematik (Crelles
  Journal)}, 2010(649):63--87, 2010.

\bibitem[Fuj87]{MR946237}
Akira Fujiki.
\newblock On the de {R}ham cohomology group of a compact {K}\"ahler symplectic
  manifold.
\newblock In {\em Algebraic geometry, {S}endai, 1985}, volume~10 of {\em Adv.
  Stud. Pure Math.}, pages 105--165. North-Holland, Amsterdam, 1987.

\bibitem[GH96]{MR1395935}
L.~G\"ottsche and D.~Huybrechts.
\newblock Hodge numbers of moduli spaces of stable bundles on {$K3$} surfaces.
\newblock {\em Internat. J. Math.}, 7(3):359--372, 1996.

\bibitem[GHJ03]{GrossHuybrechtsJoyce}
M.~Gross, D.~Huybrechts, and D.~Joyce.
\newblock {\em Calabi-{Y}au manifolds and related geometries}.
\newblock Universitext. Springer-Verlag, Berlin, 2003.
\newblock Lectures from the Summer School held in Nordfjordeid, June 2001.

\bibitem[Gri16]{MR3480704}
Julien Grivaux.
\newblock Parabolic automorphisms of projective surfaces (after {M}. {H}.
  {G}izatullin).
\newblock {\em Mosc. Math. J.}, 16(2):275--298, 2016.

\bibitem[Har77]{Hartshorne}
Robin Hartshorne.
\newblock {\em Algebraic geometry}.
\newblock Springer-Verlag, New York-Heidelberg, 1977.
\newblock Graduate Texts in Mathematics, No. 52.

\bibitem[HL10]{HuybrechtsLehn}
Daniel Huybrechts and Manfred Lehn.
\newblock {\em The geometry of moduli spaces of sheaves}.
\newblock Cambridge Mathematical Library. Cambridge University Press,
  Cambridge, second edition, 2010.

\bibitem[HT10]{MR2576800}
Brendan Hassett and Yuri Tschinkel.
\newblock Flops on holomorphic symplectic fourfolds and determinantal cubic
  hypersurfaces.
\newblock {\em J. Inst. Math. Jussieu}, 9(1):125--153, 2010.

\bibitem[Huy99]{HuyInventiones}
Daniel Huybrechts.
\newblock Compact hyper-{K}\"ahler manifolds: basic results.
\newblock {\em Invent. Math.}, 135(1):63--113, 1999.

\bibitem[Huy03]{HuyErratum}
Daniel Huybrechts.
\newblock Erratum: ``{C}ompact hyper-{K}\"ahler manifolds: basic results''
  [{I}nvent. {M}ath. {\bf 135} (1999), no. 1, 63--113; {MR}1664696
  (2000a:32039)].
\newblock {\em Invent. Math.}, 152(1):209--212, 2003.

\bibitem[Huy12]{HuybrechtsTorelli}
Daniel Huybrechts.
\newblock A global {T}orelli theorem for hyperk\"ahler manifolds [after {M}.
  {V}erbitsky].
\newblock {\em Ast\'erisque}, (348):Exp. No. 1040, x, 375--403, 2012.
\newblock S\'eminaire Bourbaki: Vol. 2010/2011. Expos\'es 1027--1042.

\bibitem[Huy16]{HuyK3}
Daniel Huybrechts.
\newblock {\em Lectures on {K}3 surfaces}, volume 158 of {\em Cambridge Studies
  in Advanced Mathematics}.
\newblock Cambridge University Press, Cambridge, 2016.

\bibitem[Joh97]{MR1472735}
D.~L. Johnson.
\newblock {\em Presentations of groups}, volume~15 of {\em London Mathematical
  Society Student Texts}.
\newblock Cambridge University Press, Cambridge, second edition, 1997.

\bibitem[Kaw97]{MR1468356}
Yujiro Kawamata.
\newblock On the cone of divisors of {C}alabi-{Y}au fiber spaces.
\newblock {\em Internat. J. Math.}, 8(5):665--687, 1997.

\bibitem[Kha02]{MR1936747}
Viatcheslav Kharlamov.
\newblock Topology, moduli and automorphisms of real algebraic surfaces.
\newblock {\em Milan J. Math.}, 70:25--37, 2002.

\bibitem[KY18]{KurnosovYasinsky}
Nikon Kurnosov and Egor Yasinsky.
\newblock Automorphisms of hyperkähler manifolds and groups acting on {CAT(0)}
  spaces.
\newblock {\em Preprint arXiv:1810.09730.}, 2018.

\bibitem[Les17]{Lesieutre2017}
John Lesieutre.
\newblock A projective variety with discrete, non-finitely generated
  automorphism group.
\newblock {\em Inventiones mathematicae}, Nov 2017.

\bibitem[Loo14]{Looijenga}
Eduard Looijenga.
\newblock Discrete automorphism groups of convex cones of finite type.
\newblock {\em Compos. Math.}, 150(11):1939--1962, 2014.

\bibitem[LOP15]{VOPsurvey}
Vladimir Lazić, Keiji Oguiso, and Thomas Peternell.
\newblock The {M}orrison−{K}awamata cone conjecture and abundance on {R}icci
  flat manifolds.
\newblock {\em Proceedings of the conference Uniformization, Riemann-Hilbert
  Correspondence, Calabi-Yau manifolds, and Picard-Fuchs Equations, Institut
  Mittag-Leffler. arXiv:1611.00556}, 2015.

\bibitem[Mar11]{Markman}
Eyal Markman.
\newblock A survey of {T}orelli and monodromy results for
  holomorphic-symplectic varieties.
\newblock In {\em Complex and differential geometry}, volume~8 of {\em Springer
  Proc. Math.}, pages 257--322. Springer, Heidelberg, 2011.

\bibitem[McM07]{McMullen07}
Curtis~T. McMullen.
\newblock Dynamics on blowups of the projective plane.
\newblock {\em Publ. Math. Inst. Hautes \'Etudes Sci.}, (105):49--89, 2007.

\bibitem[Mor96]{MR1360514}
David~R. Morrison.
\newblock Beyond the {K}\"ahler cone.
\newblock In {\em Proceedings of the {H}irzebruch 65 {C}onference on
  {A}lgebraic {G}eometry ({R}amat {G}an, 1993)}, volume~9 of {\em Israel Math.
  Conf. Proc.}, pages 361--376. Bar-Ilan Univ., Ramat Gan, 1996.

\bibitem[Muk84]{Mukai84}
Shigeru Mukai.
\newblock Symplectic structure of the moduli space of sheaves on an abelian or
  {$K3$}\ surface.
\newblock {\em Invent. Math.}, 77(1):101--116, 1984.

\bibitem[MY15]{MarkmanYoshioka}
Eyal Markman and Kota Yoshioka.
\newblock A proof of the {K}awamata-{M}orrison cone conjecture for holomorphic
  symplectic varieties of {$K3^{[n]}$} or generalized {K}ummer deformation
  type.
\newblock {\em Int. Math. Res. Not. IMRN}, (24):13563--13574, 2015.

\bibitem[Nam85]{MR771979}
Yukihiko Namikawa.
\newblock Periods of {E}nriques surfaces.
\newblock {\em Math. Ann.}, 270(2):201--222, 1985.

\bibitem[O'G97]{OGrady97}
Kieran~G. O'Grady.
\newblock The weight-two {H}odge structure of moduli spaces of sheaves on a
  {$K3$} surface.
\newblock {\em J. Algebraic Geom.}, 6(4):599--644, 1997.

\bibitem[Ogu06]{MR2231119}
Keiji Oguiso.
\newblock Tits alternative in hypek\"ahler manifolds.
\newblock {\em Math. Res. Lett.}, 13(2-3):307--316, 2006.

\bibitem[Ogu07]{MR2296437}
Keiji Oguiso.
\newblock Automorphisms of hyperk\"ahler manifolds in the view of topological
  entropy.
\newblock In {\em Algebraic geometry}, volume 422 of {\em Contemp. Math.},
  pages 173--185. Amer. Math. Soc., Providence, RI, 2007.

\bibitem[Ogu08]{Oguiso}
Keiji Oguiso.
\newblock Bimeromorphic automorphism groups of non-projective hyperk\"ahler
  manifolds---a note inspired by {C}. {T}. {M}c{M}ullen.
\newblock {\em J. Differential Geom.}, 78(1):163--191, 2008.

\bibitem[Ogu14]{MR3263669}
Keiji Oguiso.
\newblock Automorphism groups of {C}alabi-{Y}au manifolds of {P}icard number 2.
\newblock {\em J. Algebraic Geom.}, 23(4):775--795, 2014.

\bibitem[Per17]{Perego17}
Arvid Perego.
\newblock {K}{\"a}hlerness of moduli spaces of stable sheaves over
  non-projective k3 surfaces.
\newblock {\em Preprint, arXiv: 1703.02001}, 2017.

\bibitem[PS12]{MR2987307}
Artie Prendergast-Smith.
\newblock The cone conjecture for abelian varieties.
\newblock {\em J. Math. Sci. Univ. Tokyo}, 19(2):243--261, 2012.

\bibitem[PT17]{MR3621782}
Arvid Perego and Matei Toma.
\newblock Moduli spaces of bundles over nonprojective {K}3 surfaces.
\newblock {\em Kyoto J. Math.}, 57(1):107--146, 2017.

\bibitem[Ser56]{GAGA}
Jean-Pierre Serre.
\newblock G\'eom\'etrie alg\'ebrique et g\'eom\'etrie analytique.
\newblock {\em Ann. Inst. Fourier, Grenoble}, 6:1--42, 1955--1956.

\bibitem[Ste85]{Sterk}
Hans Sterk.
\newblock Finiteness results for algebraic {$K3$} surfaces.
\newblock {\em Math. Z.}, 189(4):507--513, 1985.

\bibitem[Tot10]{MR2682184}
Burt Totaro.
\newblock The cone conjecture for {C}alabi-{Y}au pairs in dimension 2.
\newblock {\em Duke Math. J.}, 154(2):241--263, 2010.

\bibitem[Tot12]{MR2931877}
Burt Totaro.
\newblock Algebraic surfaces and hyperbolic geometry.
\newblock In {\em Current developments in algebraic geometry}, volume~59 of
  {\em Math. Sci. Res. Inst. Publ.}, pages 405--426. Cambridge Univ. Press,
  Cambridge, 2012.

\bibitem[Ver13]{VerbitskyTorelli}
Misha Verbitsky.
\newblock Mapping class group and a global {T}orelli theorem for hyperk\"ahler
  manifolds.
\newblock {\em Duke Math. J.}, 162(15):2929--2986, 2013.
\newblock Appendix A by Eyal Markman.

\bibitem[Ver15]{MR3413979}
Misha Verbitsky.
\newblock Ergodic complex structures on hyperk\"ahler manifolds.
\newblock {\em Acta Math.}, 215(1):161--182, 2015.

\bibitem[Wel04]{MR2099111}
Jean-Yves Welschinger.
\newblock Deformation classes of real ruled manifolds.
\newblock {\em J. Reine Angew. Math.}, 574:103--120, 2004.

\bibitem[Yos01]{MR1872531}
K\=ota Yoshioka.
\newblock Moduli spaces of stable sheaves on abelian surfaces.
\newblock {\em Math. Ann.}, 321(4):817--884, 2001.

\end{thebibliography}

\end{document}